\documentclass[12pt,leqno]{amsart}
\usepackage[backend=biber,bibencoding=utf8,style=numeric,giveninits=true,doi=false,isbn=false,url=false,sorting=nyt]{biblatex}
\DeclareFieldFormat*{title}{#1} %
\renewbibmacro{in:}{}
\AtEveryBibitem{\clearlist{language}}
\AtEveryBibitem{%
 \clearlist{address}
 \clearfield{date}
 \clearfield{eprint}
 \clearfield{isbn}
 \clearfield{issn}
 \clearlist{location}
 \clearfield{month}
 \clearfield{series}
 \clearfield{doi}
 \clearfield{number}
 \ifentrytype{book}{}{%
  \clearlist{publisher}
  \clearname{editor}
 }
}
\usepackage[a4paper,left=24mm,right=24mm,top=32mm,bottom=30mm]{geometry}
\usepackage{csquotes}
\usepackage{epic,eepic,verbatim,amsmath,amssymb,enumitem}
\usepackage{graphicx}
\usepackage{hyperref}
\usepackage{subfigure}
\usepackage{caption}
\usepackage{tikz}
\usetikzlibrary{positioning}
\usepackage{commath}
\usepackage{esint}

\newcommand{\eps}{\varepsilon}             %
\newcommand{\R}{\mathbb{R}}                %
\newcommand{\N}{\mathbb{N}}                %
\newcommand{\M}{\mathcal{M}}
\newcommand{\PE}{\mathcal{P}^{(\eps)}}              %
\newcommand{\tPE}{\tilde{\mathcal{P}}^{(\eps)}}     %
\newcommand{\mpen}{M_\eps} %
\newcommand{\Z}{\mathbb{Z}}                %
\newcommand{\F}{\mathcal F}                %
\newcommand{\V}{\mathcal V}                %
\newcommand{\A}{\mathcal A}                %
\newcommand{\E}{\mathcal E}                %
\newcommand{\calM}{\mathcal M}             %
\newcommand{\W}{\mathcal{W}}

\newcommand{\WFu}{\mathcal{W}^{(1)}}
\newcommand{\WFv}{\mathcal{W}^{(2)}}
\newcommand{\Chi}{\mathcal{X}}

\newcommand{\LL}{\mathcal{L}}
\newcommand{\HM}{\mathcal H}               %

\newcommand{\uu}{u}
\newcommand{\uv}{v}
\newcommand{\mc}{H}
\newcommand{\ssubset}{\subset\joinrel\subset}
\DeclareMathOperator{\loc}{loc}
\DeclareMathOperator{\dist}{dist}
\renewenvironment{proof}[0] {\noindent{\em Proof.}}{\hfill \qed\medskip }

\theoremstyle{plain}
\numberwithin{equation}{section}
\newtheorem{lemma}{Lemma}[section]
\newtheorem{theorem}[lemma]{Theorem}
\newtheorem{proposition}[lemma]{Proposition}

\theoremstyle{definition}

\begin{document}

\title[diffuse willmore flow]
{A new Diffuse-interface approximation of the Willmore flow}

\author[A. R\"atz]{Andreas R\"atz}
\email{andreas.raetz@tu-dortmund.de}
\author[M. R\"oger]{Matthias R\"oger}
\email{matthias.roeger@tu-dortmund.de}



\begin{abstract}
  Standard diffuse approximations of the Willmore flow often lead to intersecting phase boundaries that in many cases do not correspond to the intended sharp interface evolution.
  Here we introduce a new two-variable diffuse approximation
  that includes a rather simple but efficient penalization
  of the deviation from a quasi-one dimensional structure of the phase fields.
  We justify the approximation property by a Gamma convergence result for the energies and a matched asymptotic expansion for the flow. Ground states of the energy are shown to be one-dimensional, in contrast to the presence of saddle solutions for the usual diffuse approximation. Finally we present numerical simulations that illustrate the approximation property and apply our new approach to problems where the usual approach leads to an undesired behavior.
\end{abstract}

\date{\today}
\maketitle

\section{Introduction}
\label{sec:intro}

Curvature energies such as the elastica energy for plane curves or the
Willmore functional for two-dimensional surfaces appear in a variety of
applications in physics, biology or image processing. Diffuse approximations
of such energies are part of many descriptions of phase transition problems
and are used as a tool for numerical simulations of corresponding sharp
interface problems. The most prominent diffuse approximation of the
Willmore and elastica energy goes back to a suggestion of De Giorgi and
builds on the well-known Cahn--Hilliard--Van der Waals functional that
represents a perimeter approximation.

The approximation property of diffuse curvature energies is closely related
to a supposed quasi one-dimensional structure of phase fields that describe
moderate-energy states. For such structures the diffuse energies represent
a certain averaging of the (sharp interface) bending energy of level
lines. However, simulations \cite{EsRR14,BrMO15} with the standard
diffuse approximations of the elastica or Willmore functional show a
somehow non-intuitive behavior and the appearance of structures where
diffuse interfaces cross. Level lines then carry an unbounded bending
energy, and the generalized sharp interface energy of such structures (in
the sense of an $L^1$ relaxation of the energy on smooth configurations)
is infinite. Such behavior can be explained by the presence of (a wealth
of) entire solutions of the stationary Allen--Cahn equation that deviate
from the quasi one-dimensional structure. Such solutions have zero diffuse
mean curvature everywhere and therefore are favored by the diffuse energies.

In many applications a preference for such structures is not consistent
with the underlying physics or the intended behavior. Therefore several
suggestions of alternative diffuse approximations have been proposed
and analyzed. However, it seems that all present approaches come with
a number of disadvantages and difficulties. In this contribution we
introduce a new approximation that uses two phase field variables and
penalizes a deviation from a quasi one-dimensional structure of the phase
fields. For each variable a standard diffuse approximation is used,
but with different choices of the double-well potential that determine
the diffuse energy functionals. As long as both phase fields retain a
quasi one-dimensional structure they are up to a simple transformation
identical. This property motivates an additional energy contribution that
penalizes a deviation from the desired behavior.

We justify this new approximation by a Gamma convergence result and show
that zero energy states for the whole space problem are, in contrast to
the classical De Giorgi approximation, necessarily one-dimensional. By a
formal asymptotic expansion we show that a suitably rescaled $L^2$ gradient
flow of the diffuse energy converges to the Willmore flow. Finally, we
present numerical simulations that demonstrate the approximation property
and consider applications of our approach in situations where the standard
approximation leads to an undesired behavior.

\bigskip
In the following we fix a nonempty open set $\Omega \subset \R^n$. Let $\M$
denote the class of open sets $E\subset\Omega $ with $\Gamma=\partial
E\cap\Omega$ given by a finite union of embedded closed
$(n-1)$-dimensional  $C^2$-manifolds without boundary in $\Omega$. We
associate to such $\Gamma$ the inner unit normal field
$\nu:\Gamma\to\R^n$,  the shape operator  $A$  with respect
to $\nu$, and  the principal curvatures $\kappa_1, \dots,
\kappa_{n-1}$ with respect to $\nu$. Finally we define the scalar mean
curvature $H=\kappa_1+\ldots + \kappa_{n-1}$ and the mean curvature
vector $\vec{H}=H\nu$,
which implies that convex sets $E$ have positive mean curvature.

The Willmore energy \cite{Willmore93} is then defined as
\begin{equation}
  \label{eq:wf}
  \mathcal{W}(\Gamma) := \frac1{2}\int_\Gamma H^2(x)\; \dif
  \mathcal{H}^{n-1}(x).
\end{equation}
Since the mean curvature vector $\vec H$ of a surface $\Gamma$
represents the $L^2$-gradient of the area functional $\A$ at $\Gamma$,
we can characterize $\W$ as the squared $L^2$-norm of the gradient of
$\A$. This observation suggests also an Ansatz for building diffuse
approximations of the Willmore functional.

The $L^2$-gradient flow of $\W$ is called Willmore flow. For an
evolving family of sets $(E(t))_{t\in(0,T)}$ in $\M$ with boundaries
$\Gamma(t)=\partial E(t)\cap \Omega$ the velocity in direction of the
inner normal field $\nu(t)$ is given by
\begin{align}
  v(t) \,&=\, -\Delta_{\Gamma(t)} {H}(t) + \frac{1}{2}H(t)^3 -
  H(t)|A(t)|^2, \label{eq:wflow}
\end{align}
on $\Gamma(t)$, where $|A(t)|^2=\kappa_1^2+\dots+\kappa_{n-1}^2$
denotes the squared Frobenius norm of the shape operator $A(t)$ and
$\Delta_{\Gamma(t)}$ denotes the Laplace--Beltrami operator on
$\Gamma(t)$.

In two dimensional space the Willmore functional for curves and the
Willmore flow are better known as Eulers elastica functional and
evolution of elastic curves. In this case \eqref{eq:wflow} reduces to
\begin{align}
  v(t)\,&=\, -\Delta \kappa_{\Gamma(t)}(t) -\frac{1}{2}\kappa(t)^3,
\end{align}
where $\kappa(t)$ denotes the curvature of $\Gamma(t)$. The Willmore
flow of a single curve in the plane exists for all times \cite{DzKS02}
and converges for fixed curve length to an elastica, see also
\cite{DAPo14,DALiPo17} for further recent results on the topic.

\subsection*{Standard diffuse approximation}
A well-known and widely used diffuse-interface approximation of the
Willmore energy builds on the Cahn--Hilliard energy
\begin{equation}
  \A_\eps(u) := \int_{\Omega} \Big(\frac{\eps}{2} |\nabla u|^2 +
  \frac{1}{\eps}W(u)\Big)\dif\LL^n,
  \label{eq:Pd}
\end{equation}
where $W$ is a suitable double-well potential and $u$ is a smooth function
on $\Omega$.

The celebrated result by Modica and Mortola \cite{MoMo77,Modi87} states
that this functionals converge, in the sense of Gamma convergence with
respect to the  $L^1(\Omega)$ topology, to a constant multiple of the
perimeter functional $\A$,
\begin{equation}
  \A_\eps \to \sigma \A,\quad
  \sigma = \int_0^1 \sqrt{2W(s)}\,ds.
  \label{eq:sigma}
\end{equation}
De Giorgi \cite{Gior91} conjectured that an approximation of the Willmore
energy is given by the squared
$L^2(\Omega)$-gradient of $\A_\eps$, integrated against the diffuse area measure. In a slightly modified form that
was introduced by Bellettini and Paolini \cite{BePa95} this leads to
the functional
\begin{align}
  \W_\eps(u) \,&:=\, \int_{\Omega} \frac{1}{2\eps} \Big(-\eps\Delta u +
  \frac{1}{\eps}W'(u)\Big)^2\dif\LL^n, \label{eq:wd}
\end{align}
that we consider in this paper as the {\em standard diffuse approximation}.
The approximation property has been confirmed in a number of
situations. Bellettini and Paolini \cite{BePa95} provided the
Gamma-$\limsup$ estimate in arbitrary dimensions. To construct a recovery
sequence for a given set $E\in\M$ two main ingredients are used: first
the optimal transition profile that connects the wells of the double-well
potential $W$, given by the unique solution $q:\R\to (0,1)$ of
\begin{equation*}
  -q'' +W'(q) =0 \quad\text{ in }\R,\quad q(-\infty)=0,\,
  q(\infty)=1,\, q(0)=\frac{1}{2}
\end{equation*}
and second the signed distance $d$ from $\Gamma$. Then, close
to $\Gamma$ the approximating phase fields are given by $u_\eps\approx
q(d/\eps)$. This is the quasi one-dimensional structure that one
might expect for phase fields with low energy values. The $\liminf$ estimate
necessary for the Gamma convergence of the diffuse Willmore approximations
turned out to be much more difficult. It was proved under several additional
assumptions in \cite{BeMu05,Mose05} and finally for dimensions $n=2,3$
and for $C^2$-regular limit points in \cite{RoeSc06}. Note that in all
these results the sum of diffuse area and diffuse Willmore functional was
considered. However, the $\liminf$-estimate for the Willmore part itself
holds as long as the diffuse surface area remains uniformly bounded.

\medskip
Building on the approximation for the Willmore functional, the corresponding
formal approximation of the Willmore flow is given by the
$L^2(\Omega)$-gradient flow of $\W_\eps$,
\begin{align}
  \eps \partial_t u \,&=\, -\Delta \big(-\eps\Delta u+
  \frac{1}{\eps}W'(u)\big) + \frac{1}{\eps^2}W''(u)\big(-\eps\Delta u+
  \frac{1}{\eps}W'(u)\big), \label{eq:Wf-d}
\end{align}
complemented by suitable boundary conditions for $u$ on
$\partial\Omega$ and an initial condition for $u$ in $\Omega$.
The convergence of the diffuse evolution towards the Willmore flow was
shown by formal asymptotic expansions by Loreti and March \cite{LoMa00}
and by Wang \cite{Wa08}. Here a smooth evolution $(\Gamma(t))_{t\in (0,T)}$
of smooth surfaces and a phase field evolution $(u_\eps(\cdot,t))_{t\in
(0,T)}$ is considered that solves \eqref{eq:Wf-d}. If one assumes that
the phase fields can be expanded around $(\Gamma(t))_t$, then the phase
fields asymptotically have the expected one-dimensional structure and
$(\Gamma(t))_{t\in (0,T)}$ evolves by Willmore flow.

\medskip
The diffuse approximation of the Willmore flow and of more general curvature
energies and flows are used for numerical simulations in a huge number
of applications. Let us only mention here
\cite{DuLiWa04,BiKaMi05,DuLiWa06,CaHeMa06,WaDu08,LoRaVo09,DoMR11,DoLeWo17,EsRR14,WaJD16,BrDM17,FrRuWi13}. For
numerical treatment of Willmore flow in a sharp-interface approach, we
refer to
\cite{BaGaNu10,BaGaNu08,BaGaNu07,Ru05,DeDzEl05,BoNoPa10,ElSt10}. Moreover,
level set techniques have been applied in order to simulate Willmore
flow in \cite{DrRu04,BeMiObSe09}.

\section*{The diffuse Willmore approximation for non-smooth limit
configurations}
\subsection*{Numerical simulations}
As mentioned above, it has been observed (see \cite{EsRR14,BrMO15} and
the references therein) that, in particular in two dimensions, simulations
based on the standard approximation lead to an in many cases undesired
behavior: when diffuse interfaces meet they tend to produce transversal
intersections of boundary layers.

This behavior is nicely illustrated if one starts from a number of equally
distributed small circles in a unit square. The elastica functional of a
ball is proportional to one over the radius, thus the balls start growing
under the Willmore (elastica) flow and touch each other in finite time. From
this point on there is no uniquely defined way how to continue the flow
and the appropriate choice might depend on the present application. In
any diffuse approximation the different balls will interact in a non-local
way. The standard diffuse Willmore flow selects an evolution where diffuse
interfaces start to flatten away from the touching points. Eventually,
a perfect checkerboard pattern develops. In the simulations, the diffuse
Willmore energy becomes extremely small in this situation.

The described behavior is in two dimensions not exceptional but rather
generic for colliding phase interfaces that evolve subject to a descent
dynamics of the standard diffuse Willmore energy. Another example occurs
in the following application that considers two phases in a fixed volume
that interact with a given inclusion. The phase is assumed to minimize
an energy consisting of a bending contribution of the phase boundary
and an adhesion energy (decreasing with increasing contact) between
one of the phases and the inclusion. In section \ref{sec:applications}
we will study this application in more detail and apply our new
diffuse approximation to this problem. Here we only consider the
evolution in the case of the standard diffuse Willmore energy and a
gradient descent method for the total energy, see
Fig.~\ref{fig:obstacle}.

\medskip
\begin{center}
  \begin{tikzpicture} %
    \node (f1) at (0,0)
    {\includegraphics*[width=0.2\textwidth]%
      {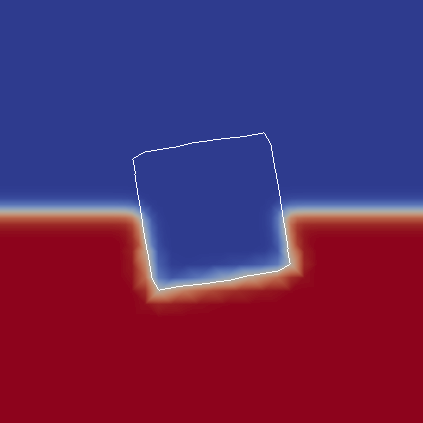}
    };
    \node (f2) [right=0.1ex of f1]
    {\includegraphics*[width=0.2\textwidth]%
      {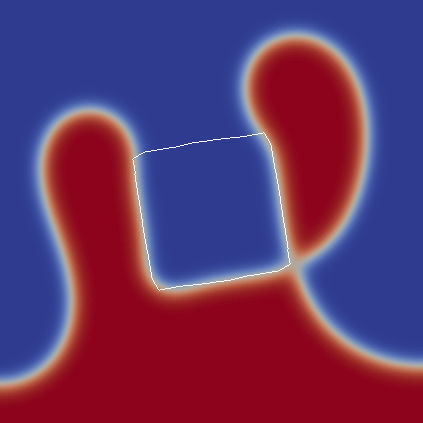}
    };
    \node (f3) [right=0.1ex of f2]
    {\includegraphics*[width=0.2\textwidth]%
      {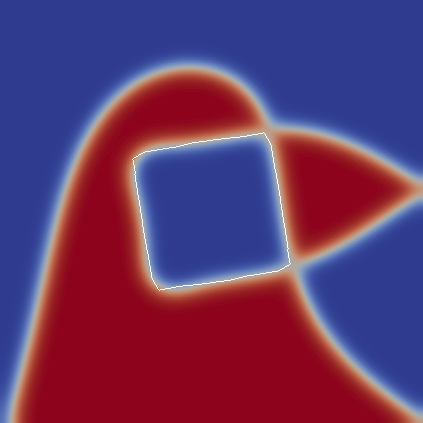}
    };
    \node (f4) [right=0.1ex of f3]
    {\includegraphics*[width=0.2\textwidth]%
      {./figures/willmoreBall147/phi1_0_5914285.png}
    };
    \node (f5) [right=0.1ex of f4]
    {\includegraphics*[width=0.07\textwidth]%
      {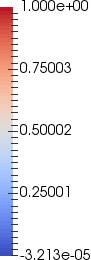}
    };
  \end{tikzpicture}
    \captionof{figure}{\footnotesize Evolution of standard flow
    with an additional adhesion to inclusion contribution: Discrete
    phase-field $u_h$ for different times   $t=0$, $t\approx0.0077$,
    \mbox{$t\approx0.0237$,} $t\approx0.5914$. The white line indicates
    the boundary of the inclusion.}\label{fig:obstacle}
\end{center}
\medskip
We clearly see that the evolution leads to configurations that do not correspond to an elastic behavior. In
particular, introducing edge like phase boundaries should not be favorable.

Such examples have already motivated several alternative diffuse
approximations. Before we comment on these and introduce our new approach
we first will discuss particular entire solutions of the Allen--Cahn
equation that promote the occurrence of intersecting phase boundaries in
diffuse approximations.

\bigskip
\subsection*{Entire solutions of the stationary Allen--Cahn equation}
The observed behavior and the occurrence of intersecting phase boundaries
does not contradict the Gamma convergence for the diffuse Willmore energy,
as all these results are restricted to $C^2$-regular interfaces and
intersections are not considered. It is an open problem to characterize the
sharp interface configurations that are in the domain of the Gamma limit of
the diffuse Willmore approximation and to characterize the Gamma limit for
non-smooth configurations. One could have guessed that the latter is given
by the $L^1$-lower-semi-continuous relaxation of the elastica energy, which
was characterized and investigated in \cite{BeDP93,BeMu07}. However, this
is not the case as can be seen from the existence of saddle solutions of the
stationary Allen--Cahn equation, which are characterized as entire solutions
that have $\frac{1}{2}$ level lines given by the union of the coordinate
axis in the plane and change sign from one quadrant to the other. The
existence of such solutions was first proved by Dang, Fife and Peletier
\cite{DaFP92} and later extended in several directions \cite{AlCM07,CaTe09}. As a consequence
of a simple spatial rescaling, this leads to a sequence $(u_\eps)_{\eps}$
of zero energy states of the standard diffuse Willmore functional that
converges to the characteristic function the first and third quadrant in
the plane. Such a configuration on the other hand has been shown to have
infinite energy with respect to the $L^1$-relaxation of the elastica energy
\cite{BeDP93}. From the observations above one may conjecture \cite{BrMO15}
that the Gamma limit of the diffuse Willmore functional is in fact given
by a generalization of the Willmore functional in the sense of integral
varifolds that behave additively on unions of one-dimensional sets. This
is in coincidence with results in \cite{Zwil18}, where intersections of
embedded curves are considered. The general case, however is open.

\subsection*{Alternative approximations that avoid intersecting phase
boundaries}
Several alternative diffuse approximations of the Willmore functional
have been introduced that avoid the occurrence of intersecting phase
boundaries. Bellettini \cite{Bell97} proposed such type of approximations
for general geometric functionals. For the Willmore functional the squared
mean curvature of the level sets of the phase field $u$ are integrated
with respect to the diffuse area density,
\begin{equation}
  \hat{\W}_\eps(u) := \frac1{2}\int_{\Omega\setminus \{|\nabla u|=0\}}
  \Big(\nabla \cdot
  \frac{\nabla u}{|\nabla u|}\Big)^2 \Big(\frac{\eps}{2} |\nabla u|^2 +
  \eps^{-1}W(u)\Big)\dif\LL^n. \label{eq:bell}
\end{equation}

An alternative approximation of the elastica functional has been
investigated by Mugnai \cite{Mugn13}, where the square integral of the
diffuse second fundamental form is considered,
\begin{equation}
  \label{eq:wplusadM}
  \bar{\W}_\eps(u) := \frac1{2 \eps}\int_\Omega\left|\eps D^2
  u - \eps^{-1}W'(u)\frac{\nabla u}{|\nabla u|} \otimes \frac{\nabla
    u}{|\nabla u|} \right|^2 \dif\LL^n.
\end{equation}

Finally, a third alternative has been introduced in \cite{EsRR14}. There,
the sum of the standard diffuse Willmore approximation and a suitable
additional energy contribution was considered that penalizes the deviation
of an appropriate rescaling of the diffuse mean curvature $w=-\eps\Delta u
+\frac{1}{\eps}W'(u)$ from the level set mean curvature $v= \nabla \cdot
\frac{\nabla u}{|\nabla u|}$. More precisely, the additional penalty term
is given
\begin{equation}
  \label{eq:energyAdd}
  \mathcal{P}_\eps(u) :=
  \frac1{2\eps^{1+\alpha}}\int_\Omega \left(w + \left(\eps|\nabla
  u|\sqrt{2W(u)}\right)^{\frac{1}{2}}v \right)^2 \dif\LL^n,
\end{equation}
where $0\leq \alpha \le 1$.

In \cite{Bell97,Mugn13,EsRR14}, for each of these proposals the
Gamma-convergence to the $L^1$-lower semicontinuous envelope of the Willmore
functional has been shown (for the $\liminf$ estimate a uniform bound on
the diffuse area is assumed). In \cite{EsRR14}, for the third proposal
numerical simulations have been included and a discussion of possible
equilibrium shapes in specific situations where discussed. Br{\'e}tin,
Masnou and Oudet \cite{BrMO15} compared the different approaches, showed
the convergence of the corresponding $L^2$-gradient flows to the Willmore
flow by formal asymptotic expansions, and discussed a number of numerical
simulations.

All three alternative approximations have the advantage that the Gamma
convergence to the $L^1$-lower semicontinuous envelope can be rigorously
shown. On the other hand, using these approaches to numerically simulate
the diffuse Willmore flow comes with several difficulties and obstacles
in all three cases, e.g. \eqref{eq:bell},\eqref{eq:energyAdd} include
the level set mean curvature which can lead to numerical difficulties,
especially due to its highly nonlinear nature and since it appears in
corresponding flows to leading order. Moreover, \eqref{eq:wplusadM}
includes the full second derivative of $u$ leading to terms which do not
have a divergence structure. We therefore
introduce in this paper a new approach that is much more easy to implement
for numerical simulations and that also seems to exclude non-generic
configurations. As a partial justification of this observation,
we prove below that zero energy states ('ground states') necessarily have
a one-dimensional structure and that Gamma-convergence in smooth points
still holds.

\section{A new diffuse Willmore flow avoiding intersections of phase
boundaries}
\label{sec: willmoreMod}

\subsection*{Doubling of variables.}
The key idea is to consider two order parameters $\uu$,
$\uv$ and diffuse Willmore energies $\W_\eps^{(1)}(\uu)$ and
$\W_\eps^{(2)}(\uv)$, where
\begin{align}
  \label{eq:We12}
  \W_\eps^{(j)}(u) := \int_\Omega\frac{1}{2\eps}
  \Big(-\eps\Delta u +
  \frac{1}{\eps}W_j'(u)\Big)^2\dif\LL^n\quad j=1,2
\end{align}
with two different double well potentials $W_1, W_2$. We then consider the
corresponding optimal profile functions $q_j:\R\to (0,1)$,
\begin{equation}
  -q_j'' + W_j'(q_j) =0,\quad \lim_{r\to -\infty}q_j(r)=0,\quad
  \lim_{r\to \infty}q_j(r)=1,\quad q_j(0)=\frac12.
  \label{eq:opt-pro}
\end{equation}
The profiles $q_j$ are strictly monotone increasing and characterized by
\begin{equation}
  q_j' = \sqrt{2W_j(q_j)},\quad q_j(0)=\frac12.
  \label{eq:opt-pro2}
\end{equation}
We then define $\Phi:(0,1)\to (0,1)$ by $\Phi=q_2\circ q_1^{-1}$ and observe
that $\Phi$ can be extended to a continuous strictly increasing and
bijective
function $\Phi:[0,1]\to [0,1]$. We obtain the properties
\begin{equation}
  q_2 = \Phi(q_1),\qquad \Phi'(q_1)= \frac{q_2'}{q_1'}
\end{equation}
and $\Phi\in C^\infty((0,1))$ for smooth double-well potentials $W_j$,
$j=1,2$.

We remark that $2W_2(q_2)= (q_2')^2 = \Phi'(q_1)^2(q_1')^2 = 2\Phi'(q_1)^2
W_1(q_1)$, which yields
\begin{equation}
  W_2\circ\Phi = (\Phi')^2W_1,\qquad W_2'\circ\Phi = 2\Phi''W_1 + \Phi' W_1'
  \label{eq:F2F1}
\end{equation}
Finally, we define $\Psi:= \Phi^{-1}$.

For given $W_1,W_2$ the transformation $\Phi$ can be determined from
\eqref{eq:opt-pro2}, which gives
\begin{equation}
  \int_{\frac{1}{2}}^{\Phi(r)}\big(2W_2(s))^{-\frac{1}{2}}\,ds
  = \int_{\frac{1}{2}}^{r}\big(2W_1(s))^{-\frac{1}{2}}\,ds.
  \label{eq:Phi}
\end{equation}
The motivation of our approach is that for quasi one-dimensional configurations
$\uu,\uv$ with nearly optimal energy with respect to the diffuse
approximations
$\W_\eps^{(1)}$ and $\W_\eps^{(1)}$, respectively, the function $\uv$
is very
close to $\Phi(\uu)$.
On the other hand, some discrepancy occurs if at least one of both is
different from the generic one-dimensional structure. A penalization of such
discrepancy prevents the phase fields from evolving to non-generic
configurations.

More specifically we introduce
for two functions $\uu, \uv$ an additional energy contribution
\begin{align}
  \PE(\uu,\uv):= \int_\Omega P\big(\Phi(\uu)-\uv\big) \dif \LL^n,
  \label{eq:penaltyEnergy}
\end{align}
where $P$ is some fixed continuous function $P:\R\to \R^+_0$ with $P(0)=0$
and $P>0$ on $\R\setminus \{0\}$.

Moreover, we choose a penalty parameter $\mpen \ge 0$
with $\mpen\to\infty$ for $\eps\to 0$. The total energy then reads
\begin{align}
  \label{eq:totalEnergy}
  \F_{\eps}(\uu,\uv) &:= \WFu_{\eps}(\uu)
  + \WFv_{\eps}(\uv)
  + \mpen\PE(\uu,\uv)\\
  &= \int_\Omega\frac{1}{2\eps}
  \Big(-\eps\Delta \uu +
  \frac{1}{\eps}W_1'(\uu)\Big)^2\dif\LL^n + \int_\Omega\frac{1}{2\eps}
  \Big(-\eps\Delta \uv +
  \frac{1}{\eps}W_2'(\uv)\Big)^2\dif\LL^n \notag\\
  &\qquad\qquad + \mpen \int_\Omega P(\Phi(\uu)-\uv) \dif \LL^n. \notag
\end{align}

\subsection*{Choices of double-well potentials and penalty term.}
We always require the following conditions of the double-well potentials
$W_1,W_2$ to be satisfied:
\begin{align}
  W_j\in C^2(\R),\quad  W_j(0)=W_j(1)=0,\quad
  W_j>0\text{ on }\R\setminus \{0,1\}, W_j''(0),W_j''(1)>0.
  \label{eq:double-well}
\end{align}
A standard example that we in particular use in our numerical simulation is
given by
\begin{equation*}
  W_1(r) = 18r^2(1-r)^2,\qquad W_2 = 4 W_1,
\end{equation*}
see Section \ref{sec:choices} below for the properties that are induced by this choice.

For the penalty energy $\PE$ we choose
$P(\Phi(\uu)-\uv) = (\Phi(\uu) - \uv)^k$ with an even number
$k=2l$, $l\in\N$. We let $k=4$ in the numerical simulations below.

For the penalty parameter we assume that
\begin{equation}
  \big(|\log\eps|q_2'(|\log\eps|)\big)^k\eps|\log\eps|M_\eps \to 0
  (\eps\to 0).
  \label{eq:ass-M}
\end{equation}
A particular choice that we consider in the asymptotic expansions and
numerical simulations below is $M_\eps=\eps^{-2}M$.

\subsection*{Evolutions laws.}
We prescribe a rescaled  $L^2(\Omega)\times L^2(\Omega)$ gradient flow
for the
total energy $\F_{\eps}$ from \eqref{eq:totalEnergy}
\begin{align}
  \label{eq:WfMod1}
  \eps \partial_t \uu&= - \nabla_{\uu}\F_{\eps}(\uu, \uv)
  = - \nabla_{\uu}\WFu_{\eps}(\uu)
  - \mpen\nabla_{\uu}\PE(\uu,\uv),\\
  \label{eq:WfMod2}
  \eps \partial_t \uv&= - \nabla_{\uv}\F_{\eps}(\uu, \uv)
  = - \nabla_{\uu}\WFv_{\eps}(\uv)
  - \mpen\nabla_{\uv}\PE(\uu,\uv),
\end{align}
where $\nabla_{\uu},\nabla_{\uv}$ denote the $L^2(\Omega)$-gradients with
respect to $\uu,\uv$ respectively.

We compute
\begin{align}
  \nabla_{\uu}\WFu_{\eps}(\uu) &= \big(-\Delta+
  \frac{1}{\eps^2}W_1''(\uu)\big)
  \big(-\eps\Delta \uu+\frac{1}{\eps}W_1'(\uu)\big),\\
  \nabla_{\uu}\WFv_{\eps}(\uu) &= \big(-\Delta+
  \frac{1}{\eps^2}W_2''(\uu)\big)
  \big(-\eps\Delta \uu+\frac{1}{\eps}W_2'(\uu)\big),\\
  \nabla_{\uu}\PE(\uu,\uv) &= -P'(\phi(\uu)-\uv)\phi'(\uu),\\
  \nabla_{\uv}\PE(\uu,\uv) &= P'(\phi(\uu)-\uv)
\end{align}
and derive the system of fourth order evolution equations
\begin{align}
  \eps\partial_t \uu &= -\big(-\Delta+ \frac{1}{\eps^2}W_1''(\uu)\big)
  \big(-
  \eps\Delta \uu+\frac{1}{\eps}W_1'(\uu)\big)+\mpen P'(\phi(\uu)-
  \uv)\phi'(\uu), \label{eq:WfMod-ex1}\\
  \eps\partial_t \uv &= -\big(-\Delta+ \frac{1}{\eps^2}W_2''(\uv)\big)
\big(\eps\Delta \uu+\frac{1}{\eps}W_2'(\uv)\big)-\mpen P'(\phi(\uu)-\uv).
 \label{eq:WfMod-ex2}
\end{align}

\section{Zero energy states in the whole space}
In this section we consider $\Omega=\R^n$ and configurations with
vanishing energy. For the standard diffuse Willmore energy we recall
that there
exist zero energy states that depend not only on one variable, such
as specific
entire solutions of the Allen--Cahn equation like the saddle solutions from
\cite{DaFP92}.

In contrast, for our total energy we prove that zero energy states are always
one-dimensional.

\begin{theorem}[Ground states are one-dimensional]
Assume that the double-well potentials $W_1,W_2$ satisfy
\eqref{eq:double-well}, and that $\{\Phi''=0\}$ is a discrete set.
Let $\eps>0$, $\Omega=\R^n$ and
consider $\uu_\eps,\uv_\eps\in H^2_{\loc}(\R^n)$, $|\uu_\eps|,|\uv_\eps|\leq
1$
with $\F_\eps(\uu_\eps,\uv_\eps)=0$.

Then one of the following two alternatives hold:
\begin{enumerate}
  \item $\uu_\eps,\uv_\eps$ are constant with
    $W_1'(\uu_\eps)=0$, $W_2'(\uv_\eps)=0$ and $\Phi(\uu_\eps)=\uv_\eps$.
  \item There exist $\nu\in\mathbb{S}^{n-1}$ and $x_0\in\R^n$ with
    $\uu_\eps(x)=q_1\big((x-x_0)\cdot\nu \big)$ and
    $\uv_\eps(x)=q_2\big((x-x_0)\cdot\nu \big)$.
\end{enumerate}
\end{theorem}
We remark that for the standard choices of $W_1,W_2$, as given above, in the
first case we have $\uu_\eps=\uv_\eps\in \{0,\frac{1}{2},1\}$.
\medskip

\begin{proof}
The property $\F_\eps(\uu_\eps,\uv_\eps)=0$ is equivalent to
\begin{equation}
  -\eps\Delta \uu_\eps+ \frac{1}{\eps}W_1'(\uu_\eps) =0,\quad
  -\eps\Delta \uv_\eps+ \frac{1}{\eps}W_2'(\uv_\eps) =0,\quad
  \uv_\eps = \Phi(\uu_\eps). \label{eq:uv}
\end{equation}
First we consider general $\uu_\eps,\uv_\eps$ not necessarily satisfying
\eqref{eq:uv} and let $\tilde{\uv}_\eps:=\Phi(\uu_\eps)$. We obtain
\begin{equation*}
  \nabla \tilde{\uv}_\eps = \Phi'(\uu_\eps)\nabla\uu_\eps,\quad
  \Delta \tilde{\uv}_\eps = \Phi''(\uu_\eps)|\nabla\uu_\eps|^2
  +\Phi'(\uu_\eps)\Delta\uu_\eps,
\end{equation*}
hence, using \eqref{eq:F2F1},
\begin{equation*}
  \frac{\eps}{2}|\nabla\tilde{\uv}_\eps|^2 +
  \frac{1}{\eps}W_2(\tilde{\uv}_\eps) =
  \Phi'(\uu_\eps)^2\Big(\frac{\eps}{2}|\nabla\uu_\eps|^2 +
  \frac{1}{\eps}W_1(\uu_\eps)\Big),
\end{equation*}
and
\begin{align}
  -\eps\Delta\tilde{\uv}_\eps + \frac{1}{\eps}W_2'(\tilde{\uv}_\eps)
  &= -\eps\Phi'(\uu_\eps)\Delta\uu_\eps
  -\eps\Phi''(\uu_\eps)|\nabla\uu_\eps|^2
    +\frac{1}{\eps}\Big(2\Phi''(\uu_\eps)W_1(\uu_\eps)
    +\Phi'(\uu_\eps)W'_1(\uu_\eps)\Big)
  \notag\\
  &= \Phi'(\uu_\eps)\Big(-\eps\Delta\uu_\eps +
  \frac{1}{\eps}W_1'(\uu_\eps)\Big)
  -2\Phi''(u_1)\Big(\frac{\eps}{2}|\nabla\uu_\eps|^2
  -\frac{1}{\eps}W_1(\uu_\eps)\Big). \label{eq:ground-0}
\end{align}
We now exploit \eqref{eq:uv}. By a spatial rescaling $x\mapsto \eps x$
we can
restrict ourselves to the case $\eps=1$ in the following.
We then deduce from \eqref{eq:uv} and \eqref{eq:ground-0} that
$\uu=\uu(\eps\cdot)$ is an entire solution of the stationary Allen--Cahn
equation,
\begin{equation}
  -\Delta \uu+ W_1'(\uu) =0
  \label{eq:ground-1}
\end{equation}
and that
\begin{equation}
  0 = \Phi''(\uu)\Big(\frac{1}{2}|\nabla\uu|^2 - W_1(\uu)\Big).
  \label{eq:ground-2}
\end{equation}
From \eqref{eq:ground-1} and elliptic regularity theory we deduce that
$u$ is smooth. Moreover, by \cite{Modi85} $\frac{1}{2}|\nabla\uu|^2 -
W_1(\uu)\leq 0$ and $u(x)=0$ or $u(x)=1$ for some $x\in \R^n$ implies
that $u$ is constant in $\R^n$. It therefore is sufficient to consider
the case $0<u<1$ on $\R^n$.

Define the set $A:=\{|\nabla\uu|^2 \neq 2W_1(\uu)\}$.
Since $\{\Phi''=0\}$ is discrete we have by \eqref{eq:ground-2} that $\uu$
is constant in each connected component of $A$.
It follows that $\nabla\uu=0$, $\Delta\uu=0$ and, by \eqref{eq:ground-1},
$W_1'(\uu)=0$ almost everywhere in $A$. By \eqref{eq:double-well} this in particular implies that $\kappa_0\leq u\leq 1-\kappa_0$ almost
everywhere in $A$ for some $0<\kappa_0<1$.

Consider any point $x\in\R^n$ such that the Lebesgue density of the set
$A$ in $x$ is not zero or one, $\theta^n(A,x)\not\in \{0,1\}$. Then we
obtain sequences $(x_k)_k$ and $(x_k')_k$ that both converge to $x$ such
that $x_k\in A$, $x_k'\not\in A$. By the preceding argument we may also
assume that $\kappa_0\leq u(x_k)\leq 1-\kappa_0$, $\nabla u(x_k)=0$ for all $k\in\N$, which implies
$\frac{1}{2}|\nabla\uu(x)|^2 - W_1(\uu(x))<0$.
On the other hand $\frac{1}{2}|\nabla\uu(x_k')|^2 - W_1(\uu(x_k'))=0$
for all $k\in\N$, hence $\frac{1}{2}|\nabla\uu(x)|^2 - W_1(\uu(x))=0$,
a contradiction.

Therefore,  $\theta^n(A,x)\in \{0,1\}$ for all $x\in\R^n$ which shows
that $A$ or $\R^n\setminus A$ have measure zero.
The second alternative yields that $u$ is constant on
$\R^n$ and satisfies the properties in item (1). Hence it remains to consider the first alternative, which implies
that $\frac{1}{2}|\nabla\uu|^2 - W_1(\uu)=0$ in $\R^n$.
By a remark in \cite{MoMo80} (see also %
\cite{CaGS94}) we obtain %
that $\uu$ is one-dimensional. In fact, define $z:\R^n\to\R$ by $\uu=q_1(z)$.
We then deduce from the properties of the optimal profile function that
\begin{equation*}
  0 = \frac{1}{2}|\nabla\uu|^2 - W_1(\uu) = \frac{1}{2}q_1'(z)^2|\nabla
  z|^2 -W_1(q_1(z)) = W(q_1(z))\big(|\nabla z|^2-1\big),
\end{equation*}
hence $|\nabla z|=1$ on $\R^n$. This further implies that
\begin{equation*}
  0 = -\Delta \uu +W_1'(\uu) = q_1''(z)|\nabla z|^2 -q_1'(z)\Delta z
  +W_1'(q_1(z)) = q_1'(z)\Delta z
\end{equation*}
in $\R^n$ and $z$ is harmonic, with uniformly bounded gradient. The
Liouville Theorem then yields that $\uu$ is a polynomial of degree
one. Since the gradient has unit length we finally deduce
$\uu(x)=q_1\big((x-x_0)\cdot\nu \big)$ for some $x_0\in\R^n$, $\nu\in
\mathbb{S}^{n-1}$.
By $\Phi(\uu)=\uv$ and the properties of $\Phi$ we also obtain that
$\uv(x)=q_2\big((x-x_0)\cdot\nu \big)$.
\end{proof}

\section{$\Gamma$ convergence of the two variable Willmore approximation}
As a consequence of \cite{RoeSc06} we obtain the $\Gamma$-convergence of our
total energy to the Willmore energy for regular configurations and
small space
dimensions.

First we extend the definition of $\F_\eps$ to the whole of
$L^1(\Omega)\times L^1(\Omega)$ by setting $\F_\eps(\uu,\uv)=\infty$ if
$\uu$ or $\uv$ does not belong to $H^2_{\loc}(\Omega)$. Next, we
introduce the following subset $\calM$ of characteristic functions
with smooth
jump set,
\begin{equation*}
  \calM := \{ u\in BV(\Omega;\{0,1\})\,:\, u=\Chi_E,\, E\ssubset\Omega
  \text{ is an open set with $C^2$-regular boundary}\}.
\end{equation*}

Moreover we define a two-variable diffuse perimeter functional $\A_\eps:
L^1(\Omega)\times L^1(\Omega)\to \R^+_0\cup\{+\infty\}$ by
\begin{equation}
  \tilde\A_\eps(\uu,\uv) := \A_\eps(\uu)+\A_\eps(\uv)
\end{equation}
if $\uu,\uv\in H^1_{\loc}(\Omega)$ and $+\infty$ else.

Finally we denote the sum of the surface tension coefficients associated to
$W_1,W_2$ by $\sigma:=\sigma_1+\sigma_2$ and define two-variable perimeter
and Willmore functionals by
\begin{align}
  \tilde\A(u,u) &:= \sigma \A(u)
  \quad\text{ if }u\in BV(\Omega;\{0,1\}),\\
  \tilde\W(u,u) &:= \sigma \W(u)\quad\text{ if }u\in\calM,
\end{align}
and by setting $\tilde\A$ and $\tilde\W$ to $+\infty$ on $\{(u,v)\in
L^1(\Omega)^2\,:\, u\neq v\text{ or }u\not\in BV(\Omega;\{0,1\})\}$ and
$\{(u,v)\in L^1(\Omega)^2\,:\, u\neq v\text{ or }u\not\in \calM\}$,
respectively.
\begin{theorem}[$\Gamma$-convergence of approximations]
\label{thm:Gamma}
Let $n=2$ or $n=3$ and consider double-well potentials $W_1,W_2$ and $P$
as  above. Then
\begin{equation}
  \Gamma-\lim_{\eps\to 0}\big(\F_\eps + \A_\eps\big)
  = \sigma \big(\tilde\W + \tilde\A\big)
\end{equation}
holds in $\calM\times\calM$.
\end{theorem}
\begin{proof}
The $\Gamma$-convergence in Theorem \ref{thm:Gamma} is equivalent to a
$\Gamma$-$\liminf$ and a $\Gamma$-$\limsup$ statement that we prove in
the next
two propositions.
\end{proof}
\begin{proposition}[$\liminf$-inequality.]
Let the assumptions of Theorem \ref{thm:Gamma} hold and consider any
sequence
$(\uu_\eps,\uv_\eps)$ in $L^1(\Omega)^2$ with
$\uu_\eps\to u$, $\uv_\eps\to v$ for some $u\in\calM$, $v\in L^1(\Omega)$. Then
\begin{equation}
  \sigma\W(u) \leq \liminf_{\eps\to 0} \F_\eps(\uu_\eps,\uv_\eps)
  \label{eq:G-liminf}
\end{equation}
holds and $u=v$ is satisfied if the right-hand side is finite.
\end{proposition}
\begin{proof}
It is sufficient to consider the case that the right-hand side of the
inequality
\eqref{eq:G-liminf} is bounded by some $\Lambda>0$. Then in
particular
\begin{equation*}
  \W_\eps^{(1)}(\uu_\eps) + \A_\eps^{(1)}(\uu_\eps) \leq \Lambda,\qquad
  \W_\eps^{(2)}(\uv_\eps) + \A_\eps^{(2)}(\uv_\eps) \leq \Lambda
\end{equation*}
holds.
We deduce from \cite{Modi87} and \cite{RoeSc06} that
\begin{align}
  \sigma_1\big(\W(u) +\A(u)\big) &\leq \liminf_{\eps\to 0}
  \W_\eps^{(1)}(\uu_\eps) + \A_\eps^{(1)}(\uu_\eps),
  \label{eq:Gamma-u}\\
  \sigma_2\big(\W(v) +\A(v)\big) &\leq \liminf_{\eps\to 0}
  \W_\eps^{(2)}(\uv_\eps) + \A_\eps^{(2)}(\uv_\eps).
  \label{eq:Gamma-v}
\end{align}
In addition, by Fatou's lemma we also have
\begin{equation*}
  \int_\Omega P(\Phi(u)-v) \leq \liminf_{\eps\to 0}\int_\Omega
  P\big(\Phi(\uu_\eps)-\uv_\eps\big)
  \leq \liminf_{\eps\to 0}\frac{1}{M_\eps}\Lambda =0,
\end{equation*}
which implies $\Phi(u)=v$. Since on the other hand $u,v\in \{0,1\}$ almost
everywhere and $\Phi(0)=0$, $\Phi(1)=1$ we obtain that $u=v$. Adding
\eqref{eq:Gamma-u} and \eqref{eq:Gamma-v} we conclude that
\eqref{eq:G-liminf}
holds.
\end{proof}

\begin{proposition}[$\limsup$-inequality]
Let the assumptions of Theorem \ref{thm:Gamma} hold and consider any
$u\in\calM$. Then there exists a sequence $(\uu_\eps,\uv_\eps)$ in
$L^1(\Omega)^2$ with $\uu_\eps\to u$, $\uv_\eps\to u$ such that
\begin{equation}
  \sigma\W(u) \geq \limsup_{\eps\to 0} \F_\eps(\uu_\eps,\uv_\eps).
  \label{eq:G-limsup}
\end{equation}
\end{proposition}
\begin{proof}
Here we consider the standard construction of recovery sequences for the
diffuse Willmore functional, see \cite{BePa95}.
Let $u\in\calM$, $u=\Chi_E$ with $E\ssubset\Omega$ open and $C^2$-regular
boundary $\Gamma$.
Denote by $d$ the signed distance function to $\Gamma$, taken positive
inside
$E$.
Next let $r_\eps:=\eps|\log\eps|$ and define $q^{(\eps)}_j:\R\to [0,1]$ by
\begin{equation*}
  q^{(\eps)}_j(r) :=
  \begin{cases}
    q_j(\eps^{-1}r) \quad&\text{ if }0\leq r\leq r_\eps,\\
    p_j^{(\eps)}(\eps^{-1}r) \quad&\text{ if }r_\eps<r<2r_\eps,\\
    1 \quad&\text{ if }r_\eps\geq 2r_\eps,\\
    1-q_j(-r) \quad&\text{ if }r<0,
  \end{cases}
\end{equation*}
where the third order polynomial $p_j^{(\eps)}$ is chosen such that
$q^{(\eps)}_j\in C^{1}(\R)$.

The approximations $\uu_\eps$, $\uv_\eps$ are then defined by
\begin{equation*}
  \uu_\eps(x) = q^{(\eps)}_1(d(x)),\quad
  \uv_\eps(x) = q^{(\eps)}_2(d(x)).
\end{equation*}
The definition of $\Phi$ yields that $\uv_\eps=\Phi(\uu_\eps)$  in
$\{|d|\leq
\eps|\log\eps|\}$ and $\{|d|\geq 2\eps|\log\eps|\}$.
Moreover, by \cite{BePa95} we have
\begin{align*}
  &\uu_\eps\to u,\quad \uv_\eps\to v\qquad\text{ in }L^1(\Omega),\\
  &\W_\eps^{(1)}(\uu_\eps) + \A_\eps^{(1)}(\uu_\eps) \,\to\,
  \sigma_1\big(\W(u)
+ \A(u)),\\
  &\W_\eps^{(2)}(\uv_\eps) + \A_\eps^{(2)}(\uv_\eps) \,\to\,
  \sigma_2\big(\W(u)
+ \A(u)).
\end{align*}
Therefore, it only remains to consider the penalty energy $\PE$. Since
$\Phi(q_1)=q_2$ and since $\Phi(0)=0$, $\Phi(1)=1$ we only have a
contribution
from the region $r_\eps<|d|<2r_\eps$, hence
\begin{align*}
  \PE(\uu_\eps,\uv_\eps) &\leq \int_{\{r_\eps<|d|<2r_\eps\}}
\big(\Phi(p_1^{(\eps)}(\eps^{-1}d(x)))-p_2^{(\eps)}(\eps^{-1}d(x))\big)^k\\
  &\leq \eps\int_\Gamma \int_{\eps^{-1}r_\eps}^{2\eps^{-1}r_\eps}
  \big(\Phi(p_1^{(\eps)}(r))-p_2^{(\eps)}(r)\big)^k
  J(y,\eps r)\,dr\,d\HM^{n-1}(y)\\
  &\qquad +\eps\int_\Gamma \int_{-2\eps^{-1}r_\eps}^{-\eps^{-1}r_\eps}
  \big(\Phi(p_1^{(\eps)}(r))-p_2^{(\eps)}(r)\big)^k
  J(y,\eps r)\,dr\,d\HM^{n-1}(y),
\end{align*}
where $J(y,s)$ denotes the Jacobi factor for the transformation
$\vartheta:\Gamma\times (-\delta,\delta)$, $\vartheta(y,s)=y+s\nu(y)$. By
the
$C^2$-regularity of $\Gamma$ we have $|J(y,s)|\leq
C$. Therefore,
\begin{align*}
  M_\eps \PE(\uu_\eps,\uv_\eps) &\leq
  2M_\eps C\A(\Gamma)r_\eps
\max_{|\log\eps|<|r|<2|\log\eps|}
\big|\Phi(p_1^{(\eps)}(r))-p_2^{(\eps)}(r)\big|^k\\
  &\leq C(\Gamma)M_\eps r_\eps
\big(q_2(\eps^{-1}r_\eps)-q_2(2\eps^{-1}r_\eps)\big)^k %
  \,\to\, 0,
\end{align*}
where we have used that $p_1^{(\eps)}$ is monotone for $\eps\ll 1$ and
assumptions \eqref{eq:ass-M} on $M_\eps$.
\end{proof}

\section{Matched asymptotic expansions}
\label{sec:asymptoticExp}

We follow and modify \cite{LoMa00}, see also \cite{Wa08,BrMO15} and the
references therein. To perform the asymptotic expansion we make here a
specific choice of the penalty term,
\begin{equation}
  \mpen = M\eps^{-2},\qquad P(r)=\frac{1}{2l}r^{2l}\quad\text{ for some
  }l\in\N, l\geq 2.
  \label{eq:mpen-choice}
\end{equation}
For the expansion we assume that the $\frac{1}{2}$-level set of the
solution $\uu_\eps(\cdot,t)$, $t\in (0,T)$ is given by a smooth evolution
of smooth hypersurfaces $(\Gamma_\eps(t))_{t\in (0,T)}$ such that at time
$t$ the set $\Omega$ is the disjoint union of $\Gamma_\eps(t)$, the open set
$\Omega^+_\eps(t)=\{\uu_\eps(\cdot,t)>\frac{1}{2}\}$,
and the open set $\Omega^-_\eps(t)=\{\uu_\eps(\cdot,t)<\frac{1}{2}\}$.
For simplicity we assume that $\Omega^+_\eps(t)\subset\Omega_0\ssubset\Omega$,
which means
that $\Gamma_\eps(t)$ does not touch $\partial\Omega$, for all $t\in (0,T)$.

We then assume that $\uu_\eps(\cdot,t)$ can be represented away from
$\Gamma_\eps(t)$ by the outer expansion
\begin{equation*}
  \uu_\eps(x,t) = \uu^{(0)}(x,t) + \eps \uu^{(1)}(x,t) +\eps^2\uu^{(2)}(x,t)
  + \ldots
\end{equation*}
and in a small neighborhood of $\Gamma_\eps(t)$ by
\begin{equation*}
  \uu_\eps(x,t) = U^{(0)}(z,y,t) + \eps U^{(1)}(z,y,t) +\eps^2U^{(2)}(z,y,t)
  + \ldots,
\end{equation*}
where $z=\frac{d_\eps(x,t)}{\eps}$ with
$d_\eps:=\dist_\eps(x,\Omega^-(t))-\dist_\eps(x,\Omega^+(t))$
the signed distance function
from $\Gamma_\eps(t)$ and $y$ the projection of $x$ on the hypersurface
$\Gamma_\eps(t)$. A similar expansion, again with respect to $\Gamma_\eps(t)$,
$t\in (0,T)$, we assume for $\uv_\eps$. Note that we do not prescribe
that $\uv_\eps(\cdot,t)=\frac12$ on $\Gamma_\eps(t)$.

Finally, we assume that $\Gamma_\eps(t)$ converges to some smooth evolution
of smooth hypersurfaces $(\Gamma(t))_{t\in (0,T)}$, more precisely that
$d_\eps = d_0 + O(\eps)$, where $d_0(\cdot,t)$ denotes the signed distance
function from $\Gamma(t)$. We will see below that only the zero order $d_0$ and geometric quantities of $\Gamma(t)$ enter up to the relevant order. Therefore, we drop in the following the explicit notation of the $\eps$-dependence of these quantities and for example write  $d$ instead of $d_\eps$.

\subsubsection*{Outer expansion}
Since we assume that $\mpen=o(\eps^{-3})$, the leading contributions
from \eqref{eq:WfMod-ex1}, \eqref{eq:WfMod-ex2} are of order $\eps^{-3}$
and imply
\begin{equation*}
  W_1''(\uu^{(0)})W_1'(\uu^{(0)}) = W_2''(\uv^{(0)})W_2'(\uv^{(0)}) =0.
\end{equation*}
A solution consistent with the boundary conditions and with the expected
transition layer structure is that away from $\Gamma(t)$
\begin{equation*}
  \uu^{(0)} = \uv^{(0)} =
  \begin{cases}
    1 \quad&\text{ in }\Omega^+(t),\\
    0 \quad&\text{ in }\Omega^-(t)
  \end{cases}
\end{equation*}
holds.
Since $\uu^{(0)}=\uv^{(0)}$, also in the next $O(\eps^{-2})$ order
no contribution from the additional penalty term appears and as in
\cite{LoMa00} we deduce that $\uu^{(1)}=\uv^{(1)}=0$. Iterating
this arguments we derive that $\uu^{(2)}=\uv^{(2)}=0$ and
$\uu^{(3)}=\uv^{(3)}=0$.

\subsubsection*{Inner expansion}
Here we also expand  the diffuse mean curvatures
\begin{equation}
  \mu^\eps_1 = -\eps\Delta\uu_\eps + \frac{1}{\eps}W_1'(\uu_\eps),\qquad
  \mu^\eps_2 = -\eps\Delta\uv_\eps + \frac{1}{\eps}W_2'(\uv_\eps)
  \label{eq:exp-mu}
\end{equation}
in a neighborhood of $\Gamma(t)$. Since we expect an (at least)
$L^2$-integrable mean curvature in the limit it is reasonable to prescribe
that they do not have a contribution of order $\eps^{-1}$ or higher, hence
\begin{equation*}
  \mu^\eps_1 = \mu_1^{(0)} + \eps\mu_1^{(1)} +
  \eps^2\mu_1^{(2)}+\ldots,\qquad
  \mu^\eps_2 = \mu_2^{(0)} + \eps\mu_2^{(1)} + \eps^2\mu_2^{(2)}+\ldots.
\end{equation*}
Since the Laplacian can be expanded in the new coordinates as
\begin{equation}
  \Delta = \frac{1}{\eps^2}\partial^2_z + \frac{1}{\eps}\Delta d \partial_z
  + \Delta_y, \label{eq:exp-Delta}
\end{equation}
see for example \cite[(2.11)]{LoMa00}, the conditions that $\mu^\eps_{1,2}$
vanish to order $\eps^{-1}$ yield
\begin{equation}
  0 = -\partial^2_{z} U^{(0)} + W_1'(U^{(0)}) = -\partial^2_{z} V^{(0)}
  + W_2'(V^{(0)}). \label{eq:U0V0}
\end{equation}
Together with the matching conditions $U^{(0)}(-\infty)=V^{(0)}(-\infty)=0$,
$U^{(0)}(\infty)=V^{(0)}(\infty)=1$ and the compatibility condition
$U^{(0)}(0,y,t)=\frac{1}{2}$, which is induced by the definition of
$\Gamma(t)$, we obtain that $\uu_\eps,\uv_\eps$ are to highest order
described by the corresponding (shifted) optimal profiles,
\begin{equation*}
  U^{(0)}(z,y,t)=q(z),\qquad V^{(0)}(z,y,t)= w\big(z - z_0(y,t)\big),
\end{equation*}
where $z_0(y,t)\in\R$ and where we set
$q=q_1$, $w=q_2$ to reduce the number of indices.
We define the linear operators
\begin{equation*}
  L_1=  -\partial^2_z + W_1''(q),\qquad
  L_2 = -\partial^2_z + W_2''(w(\cdot -z_0)),
\end{equation*}
use \eqref{eq:U0V0} and $\Delta d(x) = -\mc(y)-\eps (\sum_{i=1}^{n-1}
\kappa_i^2)z + O(\eps^2)$, see \cite{LoMa00}, and obtain
\begin{align}
  \label{eq:exp-mu-details}
  \mu^\eps_1 = &\big(L_1 U^{(1)} +\mc q'(z)\big)\\
  &+ \eps\big(L_1 U^{(2)} + \mc\partial_z U^{(1)} +(\sum_{i=1}^{n-1}
  \kappa_i^2)zq' + \frac{1}{2}W_1'''(q)(U^{(1)})^2\big)
  \notag\\
  &+ \eps^2\mu_1^{(2)} +O(\eps^3),
  \notag
\end{align}
where $q=q(z)$, $U^{(1)}=U^{(1)}(z,y,t)$, $U^{(2)}=U^{(2)}(z,y,t)$,
$\mc=\mc(y,t)$, $\kappa_i=\kappa_i(y,t)$.
A similar expansion shows for the operator $\mathcal L^\eps_1=-\eps\Delta +
\frac{1}{\eps}W_1''(u_\eps)$ that
\begin{align*}
    \mathcal L^\eps_1 &= \frac{1}{\eps}L_1 + \big(\mc\partial_z +
    W_1'''(q)U^{(1)}\big)\\
    &\qquad +\eps \big(-\Delta_y + (\sum_{i=1}^{n-1} \kappa_i^2)z\partial_z
    + W_1'''(q)U^{(2)}+\frac{1}{2}W_1^{(iv)}(q)(U^{(1)})^2\big)+ O(\eps^2).
\end{align*}
Altogether we derive for the right-hand side of \eqref{eq:WfMod-ex1} that
\begin{align}
  \label{eq:exp-ueq}
  -\frac{1}{\eps}\mathcal L_1^\eps \mu^\eps_1 &=
  -\frac{1}{\eps^2}L_1\mu_1^{(0)}
  \\
  &-\frac{1}{\eps}\big(L_1\mu_1^{(1)} + (\mc\partial_z +
  W_1'''(q)U^{(1)})\mu_1^{(0)}\big)
  \notag\\
  &-\Big[L_1\mu_1^{(2)}+\big(\mc\partial_z + W_1'''(q)U^{(1)}\big)\mu_1^{(1)}
  \notag\\
  &\qquad
  + \big(-\Delta_y + (\sum_{i=1}^{n-1} \kappa_i^2)z\partial_z +
  W_1'''(q)U^{(2)}+\frac{1}{2}W_1^{(iv)}(q)(U^{(1)})^2\big)\mu_1^{(0)}\Big]
  + O(\eps)\notag
\end{align}
Analogue expansions hold for $\uv_\eps$, where $q$ must be replaced by
$w(\cdot-z_0)$.

With these properties we obtain to leading order $\eps^0$ in
\eqref{eq:exp-mu} that
\begin{equation}
  \mu_1^{(0)} = L_1 U^{(1)} +\mc q', %
  \label{eq:mu-0}
\end{equation}
and from order $\eps^{-2}$ in \eqref{eq:WfMod-ex1} that
\begin{align}
  0 &= L_1\mu_1^{(0)}-M P'\big(\Phi(U^{(0)})-V^{(0)}\big)\Phi'(U^{(0)}).
  \label{eq:inn-exp-3}
\end{align}
We compute
\begin{align*}
  \big(\Phi(U^{(0)})-V^{(0)}\big)(z,y,t) &= \Phi(q(z))-w(z-z_0(y,t)) \\
  &= w(z) - w(z-z_0(y,t)) = z_0(y,t) b(z,y,t),\\
  \text{with }\qquad b(z,y,t) &=\int_0^1 w'\big(z-(1-r)z_0(y,t)\big)\,dr
  \,>\, 0.
\end{align*}
We multiply \eqref{eq:inn-exp-3} by $q'$ and integrate over $\R$. By
$L_1q'=0$ and the matching conditions $q'(\pm\infty)=q''(\pm\infty)=0$
we deduce that
\begin{equation*}
  0 = \int_\R P'(\Phi(q)-w)\Phi'(q)q'\,dz
  = 2lz_0^{2l-1}\int_\R b^{2l-1}\Phi'(q)q'\,dz.
\end{equation*}
Since $b,\Phi',q'$ are all positive we deduce that $z_0=0$.
In particular,
\begin{align}
  \Phi(\uu_\eps)-\uv_\eps &= \Phi(U^{(0)})-V^{(0)} + \eps
  \big(\Phi'(U^{(0)})U^{(1)} - V^{(1)}\big) + O(\eps^2)
  \label{eq:uv-diff}\\
  &= \eps\big(\Phi'(U^{(0)})U^{(1)} - V^{(1)}\big)+ O(\eps^2) \notag
\end{align}
and in \eqref{eq:inn-exp-3} the additional contribution from the penalty
term drops out. Thus we can follow \cite{LoMa00} and obtain that
\begin{equation*}
  \mu_1^{(0)} = -(\Delta d)(y,t)q'(z),\qquad U^{(1)}=0.
\end{equation*}
We now proceed similarly for $\mu_2^{(0)}$ and $V^{(1)}$ and deduce from
\eqref{eq:exp-mu-details} and \eqref{eq:WfMod-ex2}
\begin{align}
  \mu_2^{(0)} &= L_2 V^{(1)} +\mc w'(\cdot-z_0) = L_2 V^{(1)} +\mc w',
  \label{eq:mu2-0}\\
  0 &= L_2\mu_2^{(0)}+M P'\big(\phi(U^{(0)})-V^{(0)}\big) = L_2\mu_2^{(0)}.
  \label{eq:inn-exp-4}
\end{align}
Since the kernel of $L_2$ is spanned by $w'$ we deduce that
\begin{equation*}
  \mu_2^{(0)}(z,y,t) = \alpha(y,t)w'(z)
\end{equation*}
Moreover, \eqref{eq:mu2-0} yields
\begin{equation*}
  L_2 V^{(1)}(z,y,t) = \big(\alpha(y,t)+(\Delta d)(y,t)\big)w'(z).
\end{equation*}
Multiplying this equation by $w'$ and integrating in $z$ we obtain from
$L_2w'=0$ and the matching conditions at $z=\pm\infty$
\begin{equation*}
  0 = \big(\alpha(y,t)+(\Delta d)(y,t)\big) \int_\R (w')^2,
\end{equation*}
hence $\alpha(y,t)=-(\Delta d)(y,t)$ and
\begin{equation*}
  \mu_2^{(0)}(z,y,t) = -(\Delta d)(y,t)w'(z).
\end{equation*}
We also obtain $L_2 V^{(1)}=0$, which implies
\begin{equation*}
  V^{(1)}(z,y,t) = \beta(y,t)w'(z),
\end{equation*}
hence, by \eqref{eq:uv-diff} and $U^{(1)}=0$
\begin{equation}
  \Phi(\uu_\eps)-\uv_\eps  = -\eps\beta w' + O(\eps^2).
  \label{eq:pen-small}
\end{equation}
To the next order $\eps$ in $\mu_{1}^\eps$ we deduce from
\eqref{eq:exp-mu-details} and $U^{(1)}=0$ that
\begin{equation*}
  \mu_1^{(1)}(z,y,t) =  (\sum_{i=1}^{n-1} \kappa_i^2)zq'(z)+L_1U^{(2)}.
\end{equation*}
Evaluating the order $\eps^{-1}$ in \eqref{eq:exp-ueq} and using $l\geq 2$
and \eqref{eq:pen-small} we see that the penalty term
\begin{equation}
  M_\eps P'\big(\Phi(\uu_\eps)-\uv_\eps\big) =
  M\eps^{-2}\eps^{2l-1}P'\big(-\beta w'\big) + O(\eps^{4l-4})
  \label{eq:uv-diff2}
\end{equation}
does not contribute to the orders $\eps^{-1}, \eps^{0}$ and deduce that
\begin{align}
  0 &= -\mc \partial_z \mu_1^{(0)} -L_1\mu_1^{(1)}= -\mc^2q'' - L_1\mu_1^{(1)}.
  \label{eq:exp-5}
\end{align}
By \eqref{eq:uv-diff} the additional contribution from
the penalty term drops out and we can again follow \cite{LoMa00}.
Using the matching conditions
$\mu_1^{(1)}(\pm\infty,y,t)=U^{(2)}(\pm\infty,y,t)=0$ we then obtain
\begin{equation}
  \mu_1^{(1)}(z,y,t)= \mc^2(y,t)\big(\frac{z}{2}+c_1(y,t)\big)q'(z),\qquad
  U^{(2)}(z,y,t) = f_1(z,y,t)q'(z)
  \label{eq:mu1-U2}
\end{equation}
for some bounded function $c_1$ and $f_1$ given by
\begin{equation*}
  \partial_zf_1 = \frac{1}{(q')^2}g_1,\qquad
  \partial_zg_1 = \Big(\big(\sum_{i=1}^{n-1} \kappa_i^2\big)z
  -c_1 \mc^2\Big)(q')^2.
\end{equation*}
In the next step we determine the next orders in the expansion
of $\mu^\eps_2$ and \eqref{eq:WfMod-ex2}. Using that $z_0=0$ and
$V^{(0)}(z,y,t)=w(z)$ we deduce that
\begin{equation*}
  \mu_2^{(1)}(z,y,t) =  (\sum_{i=1}^{n-1}
  \kappa_i^2)zw'(z)+L_2V^{(2)}+\mc(y,t)\beta(y,t)
  w''(z)+\frac{1}{2}W_2'''(w(z))\beta^2(y,t)(w'(z))^2
\end{equation*}
and
\begin{align}
  0 &= (-\mc\partial_z -W_2'''(w)\beta w')\mu_2^{(0)} -L_2\mu_2^{(1)}
  \label{eq:exp-6}\\
  &= -\mc^2w'' -\beta \mc W_2'''(w)(w')^2-L_2\mu_2^{(1)}.
  \notag
\end{align}
Now $L_2w'=0$ implies that $L_2(w'')=-W_2'''(w)(w')^2$. We therefore have
\begin{align*}
  &\mu_2^{(1)} =  L_2\big(V^{(2)}-\frac{1}{2}\beta^2 w'' \big)+\mc\beta w'' +
  (\sum_{i=1}^{n-1} \kappa_i^2)zw',\\
  &L_2\big(\mu_2^{(1)}-\beta \mc w''\big) = -\mc^2w''.
\end{align*}
The second equation yields
\begin{equation*}
  \mu_2^{(1)}-\beta \mc w'' = \mc^2(y,t)\big(\frac{z}{2}+c_2(y,t)\big)w'(z)
\end{equation*}
for some bounded function $c_2$. Moreover
\begin{equation*}
  L_2\big(V^{(2)}-\frac{1}{2}\beta^2 w''\big) =
  \mc^2\big(\frac{z}{2}+c_2\big)w' -(\sum_{i=1}^{n-1} \kappa_i^2)zw',
\end{equation*}
which gives
\begin{equation*}
  V^{(2)}(z,y,t)-\frac{1}{2}\beta^2(y,t) w''(z) = f_2(z,y,t)w'(z)
\end{equation*}
and $f_2$ given by
\begin{equation*}
  \partial_zf_2 = \frac{1}{(w')^2}g_2,\qquad
  \partial_zg_2 = \Big(\big(\sum_{i=1}^{n-1} \kappa_i^2\big)z
  -(\frac{z}{2}+c_2)\mc^2\Big)(w')^2.
\end{equation*}
We now consider the order $O(1)$ in \eqref{eq:WfMod-ex1}, which gives
\begin{align}
  -q'\partial_t d = L_1 \mu_1^{(2)} +\mc(\mu_1^{(1)})' -
  (\Delta_y \mc)q' + \mc(\sum_{i=1}^{n-1} \kappa_i^2)zq'' +
  \mc W_1'''(q)q'U^{(2)}. \label{eq:expan-O1}
\end{align}
We multiply this equation by $q'$, integrate over $z$, and evaluate the
different terms on the right-hand side:
The matching conditions first imply
\begin{equation*}
  \int (L_1 \mu_1^{(2)})q' = 0.
\end{equation*}
Next \eqref{eq:mu1-U2} implies
\begin{equation*}
  \mc\int (\mu_1^{(1)})'q' = \mc\int \frac{1}{2}\mc^2(q')^2 +
  \mc^2\big(\frac{z}{2}+c_1\big)(\frac{1}{2}q'^2)'
  = \frac{1}{4}\mc^3 \int (q')^2
\end{equation*}
and
\begin{equation*}
  \mc(\sum_{i=1}^{n-1} \kappa_i^2)\int zq''q' =
  -\frac{1}{2}\mc (\sum_{i=1}^{n-1}
  \kappa_i^2)\int(q')^2.
\end{equation*}
Finally,
\begin{align*}
  \mc\int W_1'''(q)q'U^{(2)}q' = -\mc\int U^{(2)}L_1 q'' &= -\mc\int
  q''L_1U^{(2)}\\
  &= -\mc\int q''\mu_1^{(1)} +\mc(\sum_{i=1}^{n-1} \kappa_i^2)\int zq'q''\\
  &= \frac{1}{4}\mc^3  \int (q')^2 - \frac{1}{2}\mc(\sum_{i=1}^{n-1}
  \kappa_i^2)\int(q')^2.
\end{align*}
Therefore, we deduce from \eqref{eq:expan-O1} that
\begin{equation}
  \label{eq:evol-u}
  V = -\partial_t d = -\Delta_y \mc +\frac{1}{2}\mc^3 -\mc(\sum_{i=1}^{n-1}
  \kappa_i^2),
\end{equation}
which shows that $(\Gamma(t))_t$ evolves by Willmore flow.

To confirm the consistency of our Ansatz we also consider the order $O(1)$
in \eqref{eq:WfMod-ex2}. This gives
\begin{align*}
  -w'\partial_t d &= L_2 \mu_2^{(2)} +\mc(\mu_2^{(1)})' - (\Delta_y\mc)w'
  +\mc(\sum_{i=1}^{n-1} \kappa_i^2)zw'' +\mc W_2'''(w)w'V^{(2)}
  \notag\\
  &\quad +\beta W_2'''(w)w'\mu_2^{(1)} +\frac{1}{2}\mc W_2^{(iv)}(w)\beta^2
  (w')^3
  \notag\\
  &= L_2 \mu_2^{(2)} +\mc\big(\mc^2(\frac{z}{2}+c_2)w'\big)' - (\Delta_y \mc)w'
  +\mc(\sum_{i=1}^{n-1} \kappa_i^2)zw'' +\mc W_2'''(w)f_2(w')^2
  \notag\\
  &\quad + \beta \big(\mc^2 w'''+  \mc^2(\frac{z}{2}+c_2)W_2'''(w)(w')^2\big)
  \notag\\
  &\quad + \beta^2\big(\mc W_2'''(w)w''w'
  +\mc\frac{1}{2}W_2'''(w)w''w'+\mc\frac{1}{2}W_2^{(iv)}(w)(w')^3\big)
  \notag\\
  &= \Big[L_2 \mu_2^{(2)} +\mc\big(\mc^2(\frac{z}{2}+c_2)w'\big)'
  - (\Delta_y\mc)w' +\mc(\sum_{i=1}^{n-1} \kappa_i^2)zw'' +
  \mc W_2'''(w)f_2(w')^2\Big]
  \notag\\
  &\quad + \beta \mc^2 \big(w'''+  (\frac{z}{2}+c_2)W_2'''(w)(w')^2\big)-
  \frac{1}{2}\beta^2 \mc L_2w''',
  \notag
\end{align*}
where in the last equality we have used the identity
$-W_2^{(iv)}(w')^3=L_2w''' +3W_2'''(w)w''w'$. Integrating the equation
against $w'$ we obtain for the left-hand side and for the square bracket
on the right-hand side the analogue expressions as for the $\uu_\eps$
equation, hence
\begin{align}
  \label{eq:evol-v}
  -\partial_t d \int_\R (w')^2 &= \big[-\Delta_y \mc +\frac{1}{2}\mc^3 -
  \mc(\sum_{i=1}^{n-1} \kappa_i^2)\big]\int_\R (w')^2\\
  \notag
  &\qquad + \beta \mc^2\int_\R
  w'\Big(w'''+(\frac{z}{2}+c_2)W_2'''(w)(w')^2\Big) -\frac{1}{2}\beta^2\mc
  \int_\R w'L_2w'''.
\end{align}
Using $L_2w'=0$ the last integral on the right-hand side vanishes. Moreover,
we compute
\begin{align*}
  \int w'w'''+ \int w'(\frac{z}{2}+c_2)W_2'''(w)(w')^2 &=
  \int -(w'')^2 - \int w'(\frac{z}{2}+c_2)L_2(w'')\\
  &= \int -(w'')^2 - \int w'' L_2\big((\frac{z}{2}+c_2)w'\big) = 0
\end{align*}
since $L_2\big((\frac{z}{2}+c_2)w'\big)=-w''$.

We therefore deduce from \eqref{eq:evol-v} the same evolution law
\eqref{eq:evol-u} and hence the consistency of our Ansatz.

\section{Numerical simulations}\label{sec:applications}

In this section, we present numerical simulations of the modified two-variable diffuse Willmore flow proposed in Sec.~\ref{sec: willmoreMod}. We first consider two situations where an analytical solution is available: a growing circle in two space dimensions, and an evolution towards a configuration that is determined by minimizer of the elastica functional restricted to a suitable class of graphs. We will in both cases obtain a good agreement with the analytical solutions and therefore some justification of our approach. Moreover, we study the evolution of two colliding circles and demonstrate that in our new approach a cross formation is avoided by the additional energy contribution \eqref{eq:penaltyEnergy}.

Beyond a simple justification of our new energy, we use our approach to investigate two examples where the avoidance of intersecting phase boundaries is essential: First the example that was already brought up in the introduction (see Fig.~\ref{fig:obstacle}), namely the minimization of a functional given by the sum of elastica energy and an adhesion energy to some inclusion present in the domain. Secondly, we demonstrate that our approach can be used to approximate the value of the lower-semicontinuous envelope of the elastica energy for configurations with cusps.

\subsection{Discretization}
\label{subsec:discr}

In our numerical simulations we use suitable discretizations in time and space. For the time discretization of \eqref{eq:WfMod1},\eqref{eq:WfMod2} we
apply a semi-implicit Euler scheme, where nonlinear terms
are linearized, see \cite{EsRR14}. Moreover, we use an operator
splitting approach in order to solve each fourth order equation in
\eqref{eq:WfMod1},\eqref{eq:WfMod2} separately on the same grid. In
order to discretize in space, we introduce a triangulation $\Omega_h$
of $\Omega$ and apply linear finite elements. For the examples
presented in
Secs. \ref{subsubsec:growingCircle}--\ref{subsubsec:collCirc} we apply
uniform grids, whereas for the final two examples in
Secs. \ref{subsubsec:obstacle}--\ref{subsubsec:piri}, we have used a
simple adaptive strategy described e.g. in \cite{BaNuSt04}. The
numerical scheme is implemented in the adaptive finite element library
AMDiS \cite{amdis}.

\subsection{Choice of double-well potentials and penalty term}
\label{sec:choices}

In our numerical simulations we use the double-well potentials
\begin{equation}
  W_1(r) = 18r^2(1-r)^2,\qquad W_2 = 4 W_1. \label{eq:choiceW}
\end{equation}
The associated optimal profile functions and the surface tension
coefficients with respect to $W_j$ from \eqref{eq:sigma} are given by
\begin{equation}
  q_1(r) = \frac{e^{6r}}{1+e^{6r}},\quad \sigma_1=1,\qquad
  q_2(r) = \frac{e^{12r}}{1+e^{12r}},\quad \sigma_2=2, \label{eq:choiceW-q}
\end{equation}
hence $q_2=q_1(2\cdot)$.
We further deduce the property
\begin{equation*}
  \WFv_\eps(v) = \int_\Omega\frac{1}{2\eps}
  \Big(-\eps\Delta v +
  \frac{4}{\eps}W_1'(v)\Big)^2
  = \int_\Omega\frac{2}{\eps}
  \Big(-\frac{\eps}{2}\Delta v +
  \frac{2}{\eps}W_1'(v)\Big)^2
  =2\WFu_{\eps/2}(v),
\end{equation*}
which yields
\begin{equation}
  \frac{1}{\sigma_2}\WFv_\eps = \frac{1}{\sigma_1}\WFu_{\eps/2}
  \label{eq:choiceW-WF}
\end{equation}
and the equivalence of \eqref{eq:Wf-d} for $\eps, W_2$ and for $\tilde\eps$, $W_1$ with $\tilde\eps=\frac{1}{2}\eps$.

From \eqref{eq:choiceW-q} we deduce that
\begin{equation}
  \Phi(p) = \frac{p^2}{1-2p+2p^2}.
  \label{eq:choiceW-Phi}
\end{equation}
Moreover, we let $P(\Phi(\uu)-\uv) =
(\Phi(\uu) - \uv)^4$ and use, for better performance of the scheme, a modification of the penalty term used in the analysis above by introducing a threshold value $\theta>0$. More precisely we choose
\begin{equation}
  \label{eq:PCutOff}
  \tilde{P}(\Phi(\uu)-\uv) :=
  \left(P(\Phi(\uu)-\uv) - \theta^4\right)_+
\end{equation}
and use the modified penalty energy
\begin{align}
  \tPE(\uu,\uv):=
  \int_\Omega \tilde{P}(\Phi(u)-v) \dif \LL^n.
  \label{eq:penaltyEnergyCutOff}
\end{align}
This approach aims at locally penalizing deviations of $\uu$ and
$\uv$ from being close to optimal profiles without influencing the
flow where there is no interaction.
We remark that
the modified penalty energy \eqref{eq:penaltyEnergyCutOff} enters the
total energy \eqref{eq:totalEnergy} with the prefactor $M_\eps =
M\eps^{-2}$.

\subsection{Numerical Examples}
\label{subsec:numEx}

In general, we assume no flux boundary conditions
\begin{equation}
  \label{eq:nfbcs}
  \nabla \uu \cdot \nu_\Omega
  = \nabla \uv \cdot \nu_\Omega
  = \nabla \mu_i \cdot \nu_\Omega = 0\qquad \text{on} \quad
  \partial\Omega, \quad i=1,2,
\end{equation}
where $\mu_i$, $i=1,2$, denote the diffuse curvatures
\begin{equation*}
  \mu_1 := -\eps\Delta \uu+
  \frac{1}{\eps}W_1'(\uu), \qquad
  \mu_2 = -\eps\Delta\uv + \frac{1}{\eps}W_2'(\uv).
\end{equation*}
In addition, we use parameters from Tab. \ref{tab:parameters}. In
practice, by \eqref{eq:choiceW-q} and \eqref{eq:choiceW-WF},
one can replace $W_2$ by $W_1$ and
use $\eps_2=\frac{\eps}{2}$ for \eqref{eq:WfMod2} instead (appropriately adjusting the prefactors of the additional penalty terms).
\begin{table}[h]
  \begin{center}
    \renewcommand{\arraystretch}{1.5}
    \begin{tabular}[h]{| c | c | c | c | c | c | c |
        c | c | c | c | c | c | c | c | c }
      \hline
       parameter & $\eps$  & $M$ & $\theta$ \\
      \hline
       value &  $\frac{1}{8}$ & $10^7$ & $0.004$\\
      \hline
    \end{tabular}
    \renewcommand{\arraystretch}{1.0}
    \vskip5pt \caption{\footnotesize Parameters used for the
      simulations.}
    \label{tab:parameters}
  \end{center}
\end{table}

\subsubsection{Benchmark: Growing Circle}
\label{subsubsec:growingCircle}

As a first test, we observe that the modified flow yields a reasonable
approximation of the Willmore flow in the case of a growing circle. In
Fig.~\ref{fig:circle}, we compare numerical results with the analytic
expression for a circle growing according to Willmore flow. Thereby,
we use an initial radius $R(0)=0.25$ and plot the perimeter $L(t)=2\pi
R(t)$ of the analytic solution of the Willmore flow versus time and
compare with the discrete diffuse interface length $L_{\eps}^{(1)}$,
$L_{\eps}^{(2)}$ from
the simulation results with
\begin{align}
  \label{eq:diffPerimeter1}
  L_\eps^{(1)} (t) :=
  \int_\Omega\left(\frac{\eps}{2}|\nabla \uu(\cdot, t)|^2
    + \frac1\eps W_1(\uu(\cdot, t))
  \right)\dif\LL^n, \\
  \label{eq:diffPerimeter2}
  L_\eps^{(2)} (t) :=
  \int_\Omega\left(\frac{\eps}{2}|\nabla \uv(\cdot, t)|^2
    + \frac1\eps W_2(\uv(\cdot, t))
  \right)\dif\LL^n,
\end{align}
and with results of a standard diffuse-interface approximation,
i.e.~without the additional penalty energy contribution, see
Fig.~\ref{fig:circle}, left. The two curves for the numerical outcome
are almost indistinguishable and show the expected approximation of the
analytic curve. On the right, one can see the results for
both phase-field variables $\uu$ and $\uv$ in comparison with
the analytic solution.

\begin{figure}[h]
\begin{center}
\includegraphics*[width=0.49\textwidth]{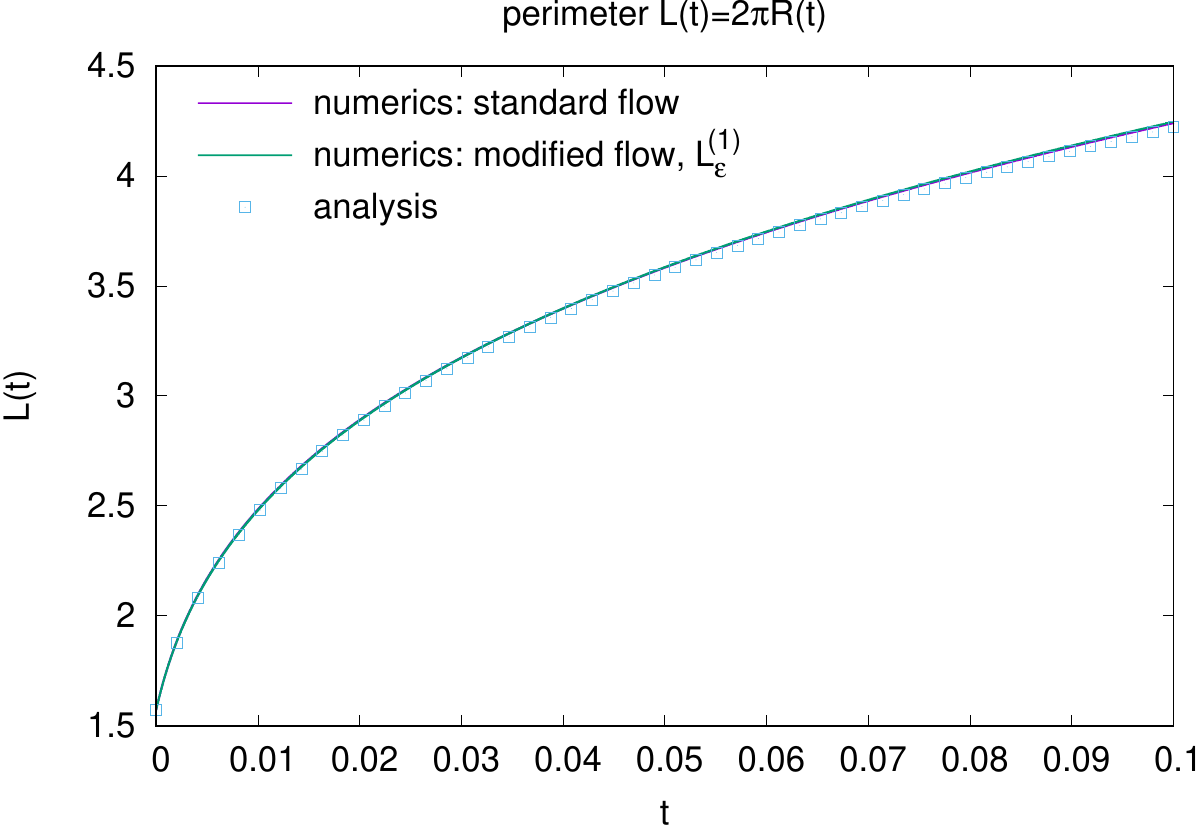}
\includegraphics*[width=0.49\textwidth]{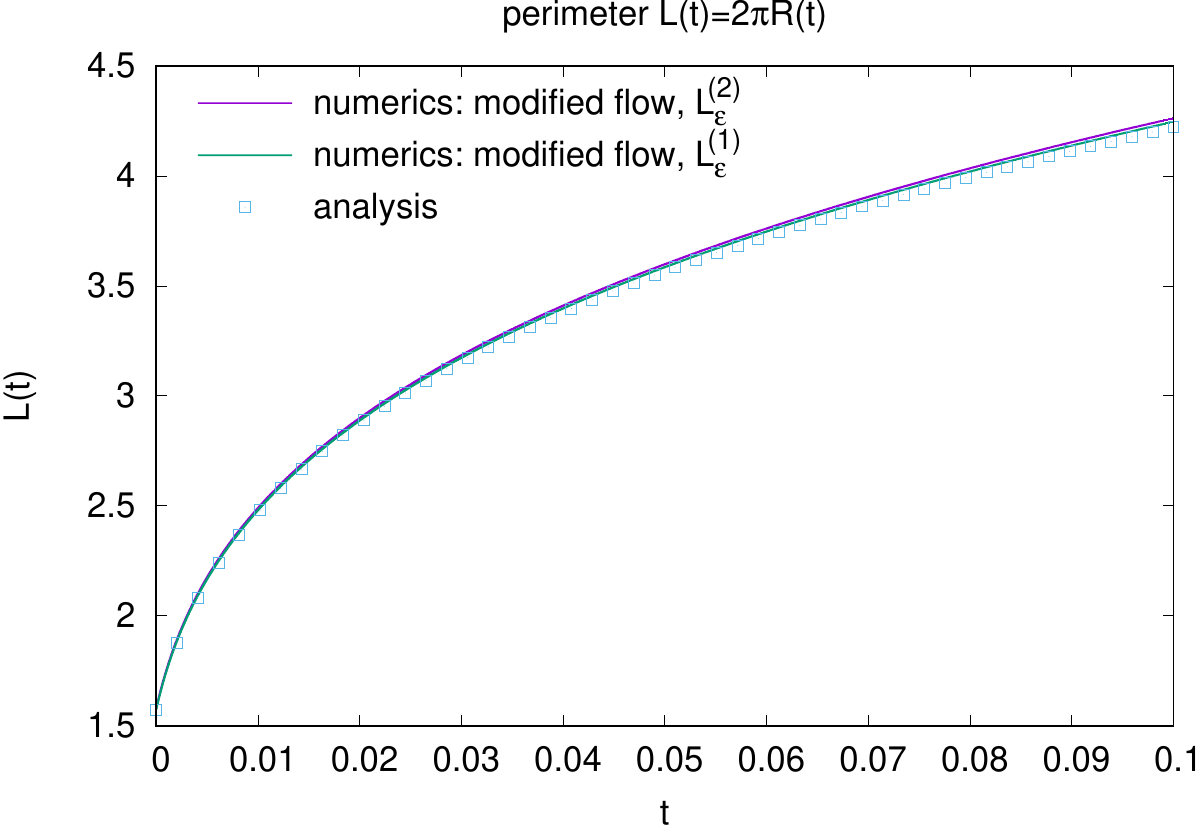}
\caption{\footnotesize \label{fig:circle} Evolution of modified flow
  \eqref{eq:WfMod1}--\eqref{eq:WfMod2}: Perimeter of growing circle
  versus time: Analytic expression and results of standard and
  modified diffuse-interface flow (left). Analytic expression and
  results of modified diffuse-interface flow for the two variables
  $\uu$ and $\uv$ (right).}
\end{center}
\end{figure}

\subsubsection{Benchmark: Example from \cite{LiJe07}}
\label{subsubsec:LiJe07}

In a second benchmark example we follow \cite{EsRR14} and we compare
in Fig.~\ref{fig:LiJe07} a nearly stationary state in numerical results
with analytic minimizers of the elastica functional found in
\cite{LiJe07}. Linn{\'e}r and Jerome consider the elastica functional
for graphs among all functions $f$ in $W^{2,2}((0,1))$ satisfying
$f(0)=0$ and $f'(1) = \infty$. Moreover, they prove existence and
uniqueness and provide an explicit representation of the
minimizer. For the numerical approach in this particular example, we
have chosen a rectangular domain $\Omega = (-1.1,1.1)\times(2.2,2.2)$
and assume periodicity of all variables on $\partial \Omega$. For our
simulations we choose initial conditions for $\uu$ and $\uv$ that
represent an ellipse. In Fig.~\ref{fig:LiJe07}, one can see the nearly
stationary level set $\{\uu_h=0.5\}$ at time $t \approx 0.2448$
compared to the analytic minimizer $\Gamma_2$ from
\cite{LiJe07}. Note that, similar to \cite{EsRR14}, we have shifted
the solution from \cite{LiJe07} in an appropriate way. Moreover, the
discrete diffuse Willmore energies $\W_{\eps}^{(1)}(\uu_h) \approx
2.82307$ and $\W^{(2)}_{\eps}(\uv_h)\approx 2.88127$ are close to
the analytic value $\approx 2.8711$. Together with the benchmark
problem from Sec.~\ref{subsubsec:growingCircle} we can conclude that
the modified flow \eqref{eq:WfMod1},\eqref{eq:WfMod2} yields a sound
quantitative approximation of Willmore flow of curves.
\begin{figure}
\begin{center}
\includegraphics*[width=0.3\textwidth]{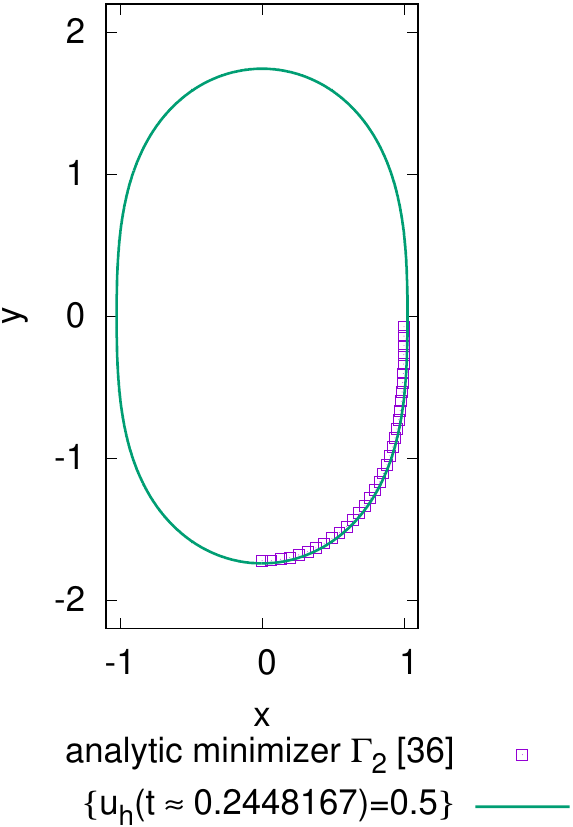}
\caption{\footnotesize \label{fig:LiJe07} Almost stationary soultion
  of \eqref{eq:WfMod1}--\eqref{eq:WfMod2}. Level curve $\{\uu_h=1/2\}$
  at time $t \approx 0.2448$ and shifted version $\Gamma_2$ of the
  analytic minimizer from \cite{LiJe07}.}
\end{center}
\end{figure}

\subsubsection{Colliding Circles}
\label{subsubsec:collCirc}

In Fig.~\ref{fig:mod2}, we consider an initial condition
representing two circles. During the evolution of the modified flow
the circles grow until the interfaces come sufficiently close, and
unlike for the standard diffuse Willmore flow we do not observe any
transversal intersections as reported in \cite{EsRR14,BrMO15}. In
contrast, the interfaces stop moving at the meeting point and one can
see the evolution of $\uu_h$ and $\uv_h$ towards almost stationary
discrete states in Fig.~\ref{fig:mod2}.

\begin{figure}
\begin{center}
\includegraphics*[width=0.17\textwidth]{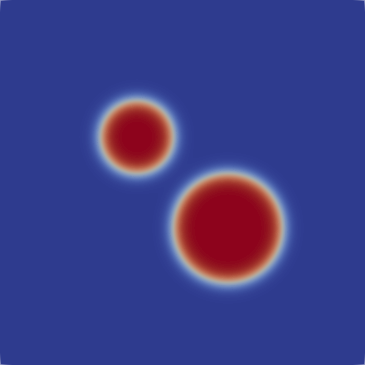}
\includegraphics*[width=0.17\textwidth]{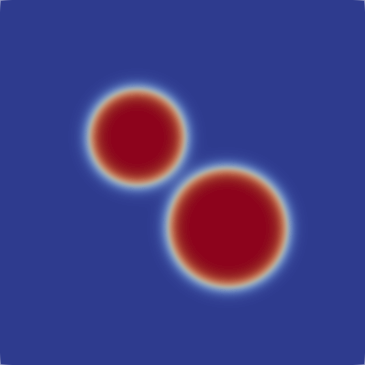}
\includegraphics*[width=0.17\textwidth]{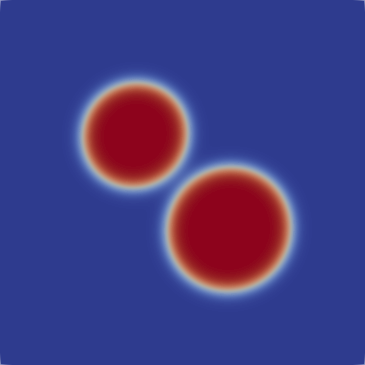}
\includegraphics*[width=0.17\textwidth]{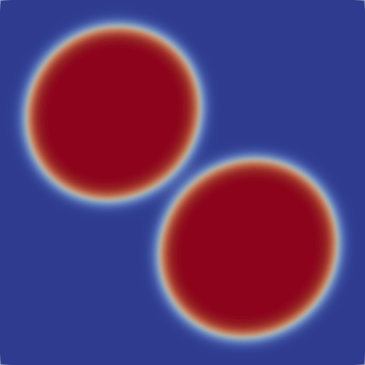}
\includegraphics*[width=0.17\textwidth]{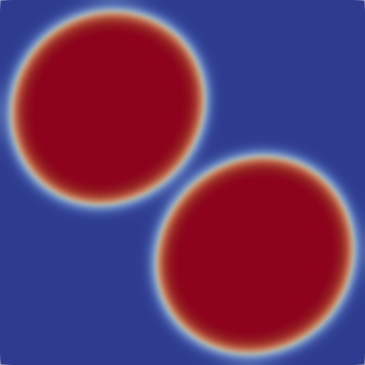}
\includegraphics*[width=0.06\textwidth]{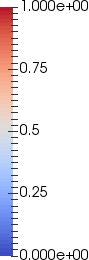}\\
\includegraphics*[width=0.17\textwidth]{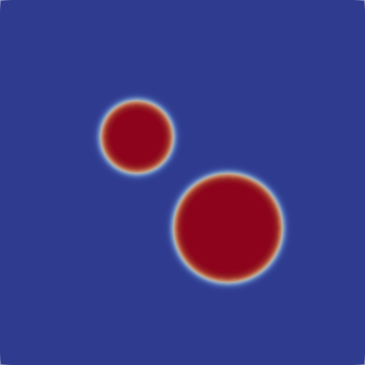}
\includegraphics*[width=0.17\textwidth]{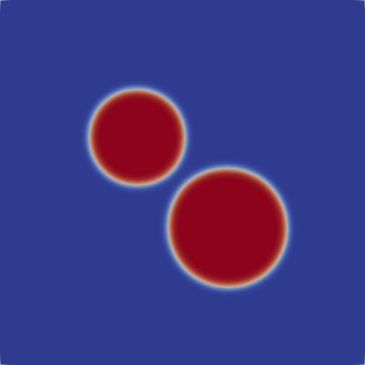}
\includegraphics*[width=0.17\textwidth]{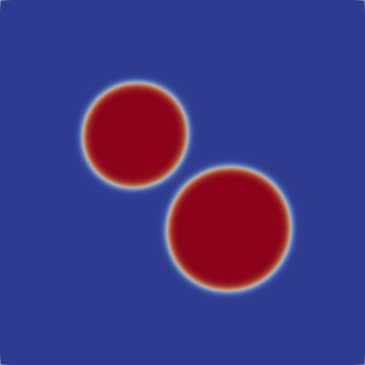}
\includegraphics*[width=0.17\textwidth]{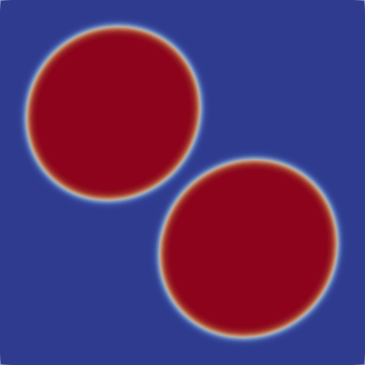}
\includegraphics*[width=0.17\textwidth]{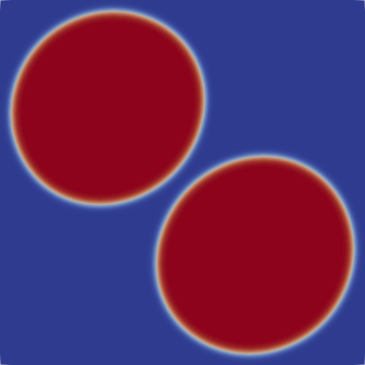}
\includegraphics*[width=0.06\textwidth]{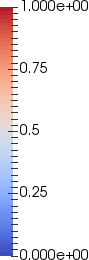}
\caption{\footnotesize \label{fig:mod2} Evolution of modified flow
  \eqref{eq:WfMod1}--\eqref{eq:WfMod2}: Discrete phase-fields
  $\uu_h$ and $\uv_h$ for
  different times $t=0$, $t\approx0.0016$, $t\approx0.0030$, \mbox{$t\approx0.0362$,}
  $t\approx0.1285$.}
\end{center}
\end{figure}

\subsubsection{Adhesion to a domain inclusion}
\label{subsubsec:obstacle}

We come back to the motivating example from Fig.~\ref{fig:obstacle}
and use the approach presented in this contribution in order to avoid
transversal intersections which are undesired from the application
point of view. From the modeling perspective, we follow
\cite{DaAuGo14} and consider a domain $\Omega\subset \R^n$ with a particle inclusion, which is given by an open set $\Omega_p\ssubset\Omega$. We introduce a sharp-interface membrane energy
\begin{equation}
  \label{eq:obstacleEnergySharp}
  \E(\Gamma) :=   \mathcal{W}(\Gamma)
  - \underbrace{w\int_{\Gamma_{\text{ad}}} \; \dif
  \mathcal{H}^{n-1}(x)}_{=:\E_{\text{ad}}}
\end{equation}
that includes an elastic contribution and an additional contact adhesion energy $\E_{\text{ad}}$. The latter measures the size of the contact set $\Gamma_{\text{ad}}:=\Gamma\cap \partial\Omega_p$, the adhesion strength is determined by the parameter $w\ge 0$. In addition, admissible membranes are confined to the set $\Omega\setminus\Omega_p$.

In order to obtain a diffuse-interface counterpart of
\eqref{eq:obstacleEnergySharp}, we introduce variables
$\psi_\eps^{(i)}:\Omega \to \R$, $i=1,2$,
\begin{equation}
  \label{eq:obstaclePsi}
  \psi^{(1)}_\eps(x) := q_1\Big(\frac{d_p}{\eps}\Big), \quad
  \psi^{(2)}_\eps(x) := q_2\Big(\frac{d_p}{\eps}\Big), \quad x \in \Omega,
\end{equation}
with $d_p$ the signed
distance to $\partial \Omega_p$, where $d_p>0$ in $\Omega_p$. The
diffuse membrane energy then reads
\begin{equation}
  \label{eq:obstacleEnergyDiff}
  \E_\eps(u,v) := \F_\eps(u,v) -  \E^{(1)}_{\text{ad}, \eps}(\uu)
  - \E^{(2)}_{\text{ad}, \eps}(\uv)
\end{equation}
with
\begin{equation*}
  \E^{(i)}_{\text{ad}, \eps}(u) := \int_\Omega \left(
    \frac{\eps}{2}|\nabla u|^2 + \eps^{-1}W_i(u)
  \right)2\eps^{-1}W_i(\psi_\eps^{(i)}), \quad i=1,2.
\end{equation*}
Moreover, in order to account for a volume
constraint, we add a penalty energy to \eqref{eq:obstacleEnergyDiff}
penalizing deviations from a prescribed diffuse volume integral
value. Finally, we include another penalty energy contribution
\begin{equation*}
  \sim\int_\Omega (\psi^{(1)}_\eps)^2(1-\uu-\psi_\eps^{(1)})^2
  +\int_\Omega (\psi^{(2)}_\eps)^2(1-\uv-\psi_\eps^{(2)})^2
\end{equation*}
which prevents the interface from going through the domain inclusion,
where $\psi_\eps \approx 1$. In Fig.~\ref{fig:modObstacle} we see the
results for the corresponding flow. For this particular example an
adaptively refined grid has been applied, where the grid is locally refined or coarsened according to the
values of $\uu_h$, $\uv_h$ and $\psi_\eps$, see \cite{BaNuSt04}. One
observes that the formation of transversal interfaces as in
Fig.~\ref{fig:obstacle} is now prevented by the new energy
modification. For more detailed information, we refer to \cite{RaRo19}.

\begin{figure}
\begin{center}
  \includegraphics*[width=0.2\textwidth]{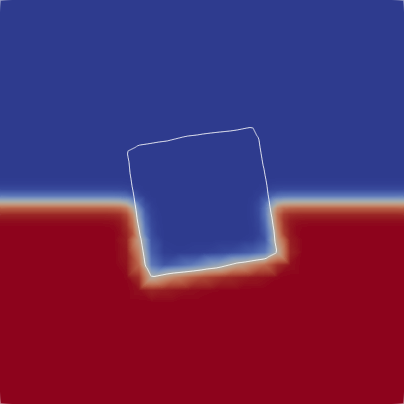}
  \hspace{0.01\textwidth}
  \includegraphics*[width=0.2\textwidth]{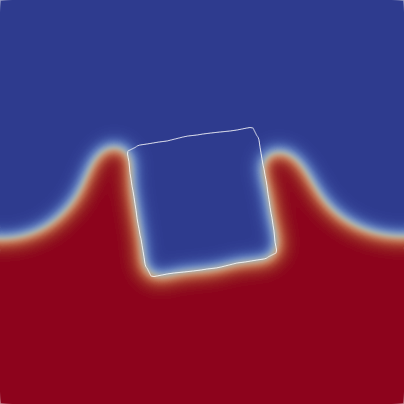}
  \hspace{0.01\textwidth}
  \includegraphics*[width=0.2\textwidth]{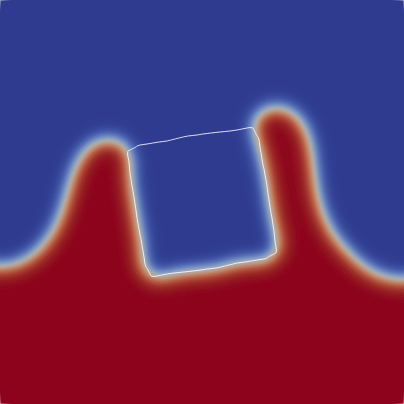}
  \hspace{0.01\textwidth}
  \includegraphics*[width=0.2\textwidth]{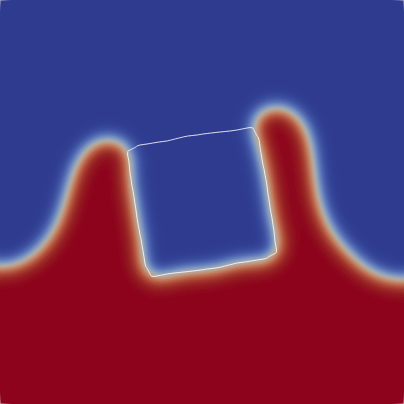}
  \hspace{0.01\textwidth}
  \includegraphics*[width=0.07\textwidth]{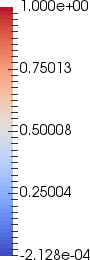}
  \caption{\footnotesize \label{fig:modObstacle} Evolution of modified
    flow for adhesion to domain inclusion application: Discrete
    phase-field $\uu_h$ for different times $t=0$, $t\approx0.0011$,
    \mbox{$t\approx0.0340$,} $t\approx0.3678$.}
\end{center}
\end{figure}

\subsubsection{Approximation of the lower-semicontinuous envelope
of the elastica functional}
\label{subsubsec:piri}

Here we consider a particular set $E$ in the plane resembling a cloverleaf. The set has four connected components, each with a simple cusp. Due to the non-smooth boundary, the elastica functional is not defined for such a configuration. To associate an elastic energy to $E$ one can evaluate the value of the lower-semicontinuous envelope, given by
\begin{equation*}
  \overline{\W}(E) = \inf \{\liminf_{k\to\infty} \W(E_k)\},
\end{equation*}
where the infimum is taken over all sequences $(E_k)_k$ of sets $E_k\ssubset\Omega$ with $C^2$-boundary converging in an $L^1(\Omega)$ sense to $E$. The lower-semicontinuous envelope of the sum of area and elastica functional was studied in \cite{BeDP93} and \cite{BeMu04,BeMu07}, where also the cloverleaf example was considered. It was shown that for all configurations with an even number of simple cusps $\overline{\W}$ is finite. Moreover the minimization procedure in the definition of $\overline{\W}$ leads to an optimal system of curves that extends the boundary $\partial E$ by `ghost interfaces' with even multiplicity. Here we consider a similar problem, where the minimization of the sum of length and elastica functional is replaced by the minimization of the elastica functional subject to a confinement constraint (to the computational domain $\Omega$). To obtain a diffuse approximation a modification of the classical approach is essential, since the latter in general leads to limit configurations with transversal intersections that carry an energy $\overline{\W}$ that is unbounded. We demonstrate here that our modified diffuse Willmore flow leads to reasonable results and configurations that correspond to an optimal system of curves arising in the minimization procedure associated to the definition of $\overline{\W}$ (even if we cannot provide a rigorous justification in the sense of convergence of our modified energy to the lower-semicontinuous envelope of the elastica functional).

With this aim let $E \ssubset \Omega$ be given by four symmetrically distributed drops described by so-called piriforms, see \cite{La72}. We introduce an additional energy contribution that penalizes deviation of the diffuse fields from the set $E$, given by
\begin{equation*}
  \F_{\psi_\eps^{(1)}}(\uu) + \F_{\psi_\eps^{(2)}}(\uv),\qquad
  \F_{\psi_\eps^{(i)}}(u) = \frac{M_\psi^\eps}{2}\int_\Omega (u - \psi_\eps^{(i)})^2,
\end{equation*}
where we choose
\begin{align*}
  \psi_\eps^{(i)}(x) := q_i\left(\frac{d(x)}{\eps}\right), \quad x \in \Omega,
\end{align*}
where $d$ denotes the signed distance to $\partial E$, with
$d>0$ in $E$.
In Fig.~\ref{fig:piriPsi1}, one can see a contour plot of
$\psi^{(1)}_{\eps,h}$ used for the simulation results in
Fig.~\ref{fig:piri2}, where the evolution of $\uu_h$ for
the standard flow with $\eps=\frac1{16}$ towards an almost stationary
state is shown. Starting from a circular interface one observes that the the discrete solutions $\uu_h$ approach the characteristic function of $E$. In the final (almost stationary) configuration  the four cusps are connected by straight diffuse layers with transversal crossings. The computation of the
discrete Willmore energy for the almost stationary state yields
\begin{equation}
  \label{eq:discrWE1}
  \W_{\frac{1}{16}}(\uu_h)  \approx 57.6219.
\end{equation}
A numerical integration of the Willmore energy for the analytic
parameterization of the four piriforms gives
\begin{equation}
  \label{eq:analyticWE1}
  \W(\Gamma)  \approx  64.5136.
\end{equation}
The lower value of the approximate energy is due to the fact that for positive $\eps>0$ an asymptotically small deviation of the diffuse phases from the cloverleaf is allowed. For smaller $\eps$ values one obtains discrete diffuse-interface
Willmore energies
\begin{equation*}
  \W_{\frac{1}{32}}(\uu_h)  \approx 61.7408
  \quad \text{and} \quad
  \W_{\frac{1}{64}}(\uu_h)  \approx 64.3346,
\end{equation*}
which shows the excellent approximation of the analytic value
\eqref{eq:analyticWE1} for decreasing $\eps$, see also plots of
almost stationary solutions to the standard flow for
$\eps=\frac{1}{32}$ and $\eps=\frac{1}{64}$ in Fig.~\ref{fig:piri2Stat}.
\begin{figure}
\begin{center}
  \includegraphics*[width=0.2\textwidth]{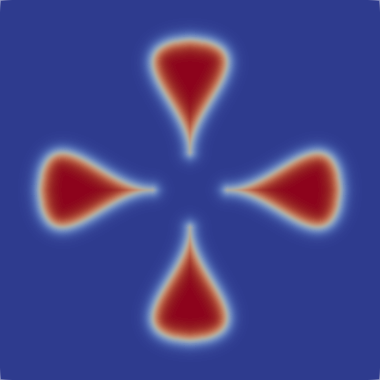}
  \hspace{0.01\textwidth}
  \includegraphics*[width=0.07\textwidth]{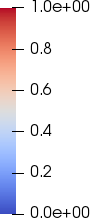}
  \caption{\footnotesize \label{fig:piriPsi1} Contour plot of $\psi_{\eps,h}^{(1)}$.}
\end{center}
\end{figure}

\begin{figure}
\begin{center}
  \includegraphics*[width=0.2\textwidth]{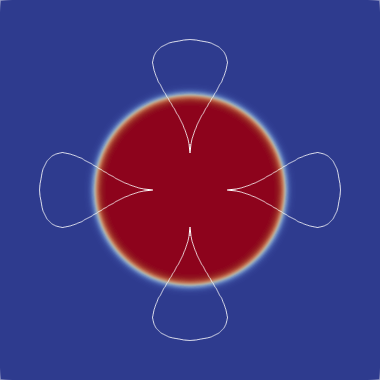}
  \hspace{0.01\textwidth}
  \includegraphics*[width=0.2\textwidth]{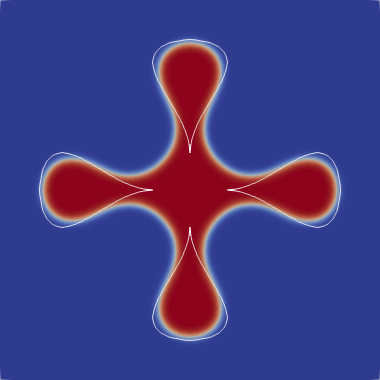}
  \hspace{0.01\textwidth}
  \includegraphics*[width=0.2\textwidth]{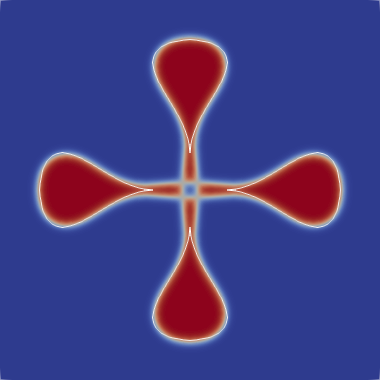}
  \hspace{0.01\textwidth}
  \includegraphics*[width=0.2\textwidth]{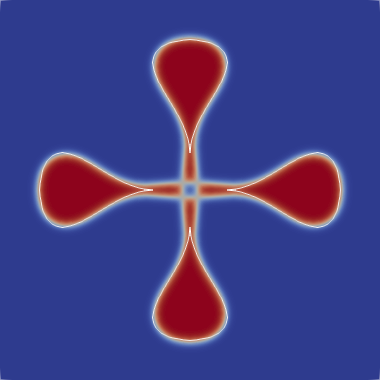}
  \hspace{0.01\textwidth}
  \includegraphics*[width=0.07\textwidth]{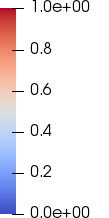}
  \caption{\footnotesize \label{fig:piri2} Evolution of standard flow
    ($\eps=\frac{1}{16}$): Discrete phase-field $\uu_h$ for different times
    $t=0$, $t\approx0.00020$, \mbox{$t\approx0.00485$,} $t\approx0.19485$.}
\end{center}
\end{figure}
\begin{figure}
\begin{center}
  \includegraphics*[width=0.2\textwidth]{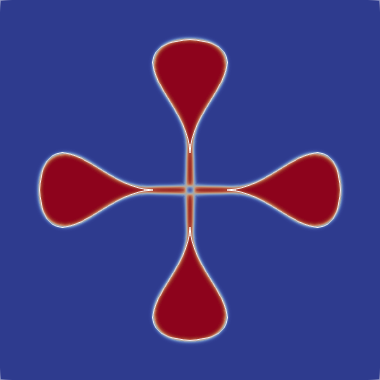}
  \hspace{0.01\textwidth}
  \includegraphics*[width=0.2\textwidth]{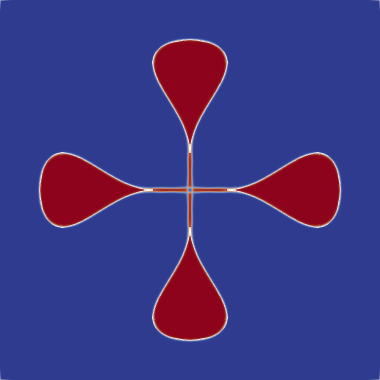}
  \hspace{0.01\textwidth}
  \includegraphics*[width=0.07\textwidth]{./figures/willmoreModA887/phi2Legend.png}
  \caption{\footnotesize \label{fig:piri2Stat} Almost stationary discrete
    solutions $\uu_h$ of standard flow ($\eps=\frac{1}{32}$, left;
    $\eps=\frac{1}{64}$, right).}
\end{center}
\end{figure}
The results for the corresponding simulation of the modified flow are
displayed in Fig.~\ref{fig:piri1Mod}. The method prohibits the
formation of the transversal crossings observed for the standard
flow. Instead, much thicker connections between the cusps are formed and this connections give a positive contribution to the Willmore energy. The discrete Willmore energy in this case becomes
\begin{equation}
  \label{eq:discrWE2}
  \W_{\frac{1}{16}}(\uv_h)  \approx 79.4029
\end{equation}
substantially increased compared to \eqref{eq:discrWE1}. This is even more significant as the configuration still deviates from the set $E$ and hence takes more freedom to minimize the elastic energy.

\begin{figure}
\begin{center}
  \includegraphics*[width=0.2\textwidth]{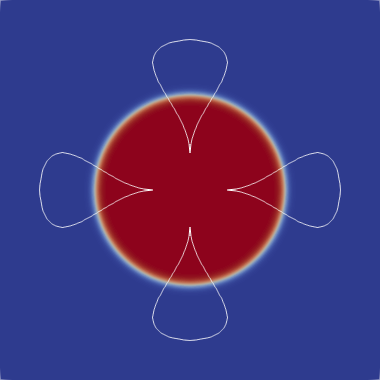}
  \hspace{0.01\textwidth}
  \includegraphics*[width=0.2\textwidth]{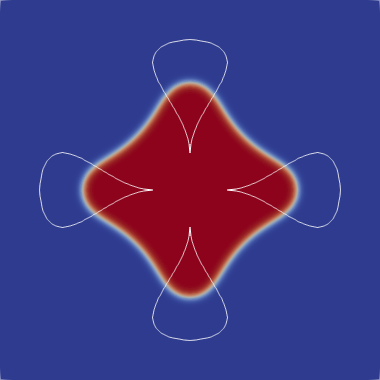}
  \hspace{0.01\textwidth}
  \includegraphics*[width=0.2\textwidth]{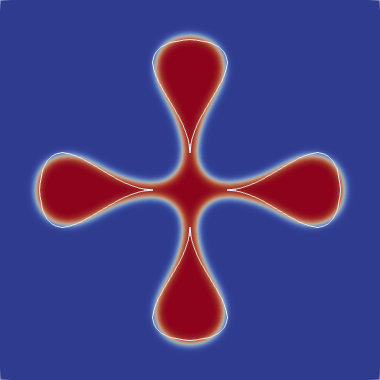}
  \hspace{0.01\textwidth}
  \includegraphics*[width=0.2\textwidth]{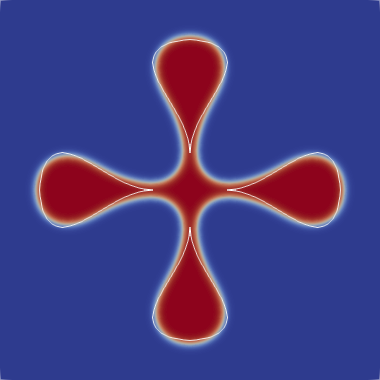}
  \hspace{0.01\textwidth}
  \includegraphics*[width=0.07\textwidth]{./figures/willmoreModA887/phi2Legend.png}
  \caption{\footnotesize \label{fig:piri1Mod} Evolution of modified flow
    ($\eps=\frac{1}{8}$): Discrete phase-field $\uv_h$ for different times
    $t=0$, $t\approx0.00019$, \mbox{$t\approx0.01096$,} $t\approx0.06896$.}
\end{center}
\end{figure}

For the modified flow with a two-component initial condition, one
obtains in Fig.~\ref{fig:piriMod2} an evolution
towards an almost stationary state approximating the piriforms, where
two components are connected by a quarter circle, see
Fig.~\ref{fig:piriMod2_2} for almost stationary states for
approximations with reduced $\varepsilon$ values. The computation of
the discrete Willmore energies then gives
\begin{align}
  \label{eq:discrWE3}
  &\W_{\frac{1}{16}}(\uv_h)  \approx 56.9323, \;
  \W_{\frac{1}{32}}(\uv_h)  \approx 64.2198, \\ \nonumber
  &\W_{\frac{1}{64}}(\uv_h)  \approx 68.8152, \;
  \W_{\frac{1}{128}}(\uv_h)  \approx 72.4494, \;
  \W_{\frac{1}{256}}(\uv_h)  \approx 74.9083,
\end{align}
and the analytic expression of the sharp-interface limit yields
\begin{equation*}
  \W(\Gamma)  \approx 81.0483,
\end{equation*}
where one has to add the Willmore energy of four quarter circles to
the energy of the piriforms in \eqref{eq:analyticWE1}. The simulation
of the standard flow with the same initial conditions as in
Fig.~\ref{fig:piriMod2} leads to results presented in
Fig.~\ref{fig:piri2Stand}. One observes an evolution with intermediate
transversal crossings.

In conclusion we see that the minimization with the standard diffuse Willmore energy does not lead to configurations that attain the minimal energy with respect to the lower semi-continuous envelope. Instead, the cusps are connected with straight ghost interfaces that transversally intersect and have infinite energy $\overline\W$. In contrast, the modified energy leads to optimal configurations that are expected as the minimizing systems of curves in the characterization of $\overline\W$. In both cases, due to the presence of a large number of local minimizer, the stationary states of the corresponding gradient flows depend very much on the initial states. In order to find the minimal energy configurations sophisticated guesses are required.

\begin{figure}
\begin{center}
  \includegraphics*[width=0.2\textwidth]{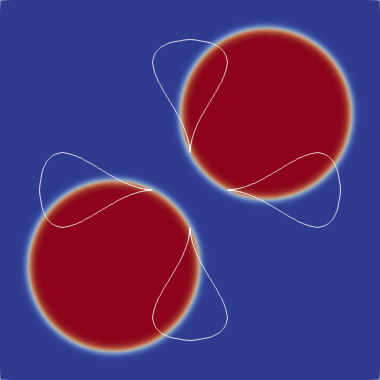}
  \hspace{0.01\textwidth}
  \includegraphics*[width=0.2\textwidth]{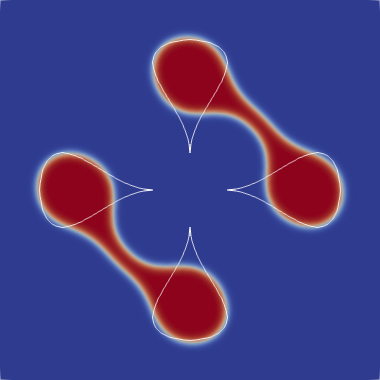}
  \hspace{0.01\textwidth}
  \includegraphics*[width=0.2\textwidth]{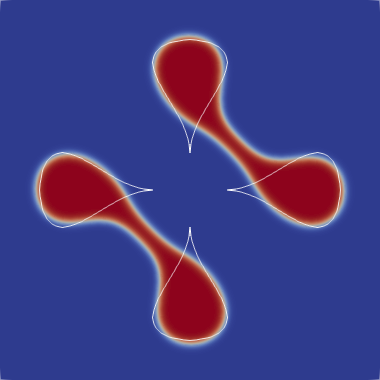}
  \hspace{0.01\textwidth}
  \includegraphics*[width=0.2\textwidth]{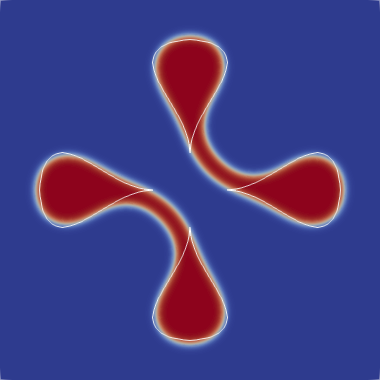}
  \hspace{0.01\textwidth}
  \includegraphics*[width=0.07\textwidth]{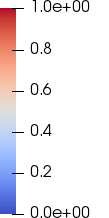}
  \caption{\footnotesize \label{fig:piriMod2} Evolution of modified flow
    ($\eps=\frac{1}{8}$): Discrete phase-field $\uv_h$ for different times
    $t=0$, $t\approx0.00102$, \mbox{$t\approx0.01002$,} $t\approx0.18402$.}
\end{center}
\end{figure}

\begin{figure}
\begin{center}
  \includegraphics*[width=0.2\textwidth]{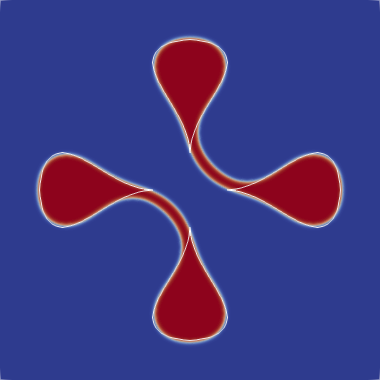}
  \hspace{0.01\textwidth}
  \includegraphics*[width=0.2\textwidth]{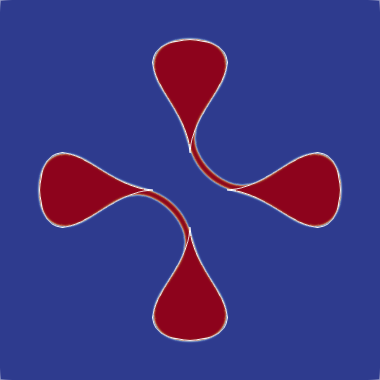}
  \hspace{0.01\textwidth}
  \includegraphics*[width=0.2\textwidth]{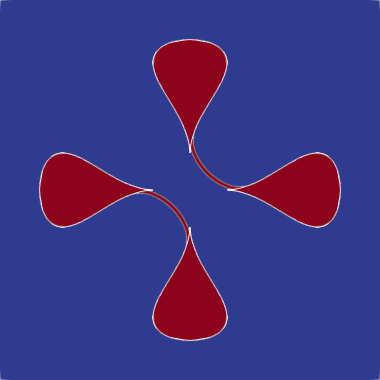}
  \hspace{0.01\textwidth}
  \includegraphics*[width=0.2\textwidth]{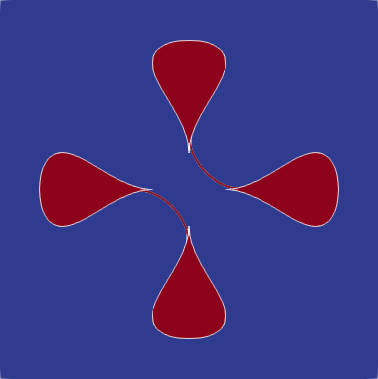}
  \hspace{0.01\textwidth}
  \includegraphics*[width=0.07\textwidth]{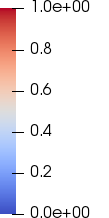}
  \caption{\footnotesize \label{fig:piriMod2_2} Almost stationary
    discrete solutions $\uv_h$ of modified flow for $\varepsilon =
    \frac{1}{16}$, $\varepsilon = \frac{1}{32}$, $\varepsilon =
    \frac{1}{64}$ and $\varepsilon = \frac{1}{128}$ (from left to right).}
\end{center}
\end{figure}

\begin{figure}
\begin{center}
  \includegraphics*[width=0.2\textwidth]{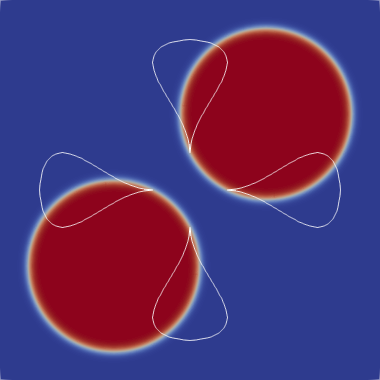}
  \hspace{0.01\textwidth}
  \includegraphics*[width=0.2\textwidth]{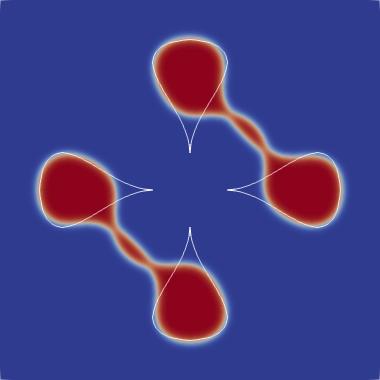}
  \hspace{0.01\textwidth}
  \includegraphics*[width=0.2\textwidth]{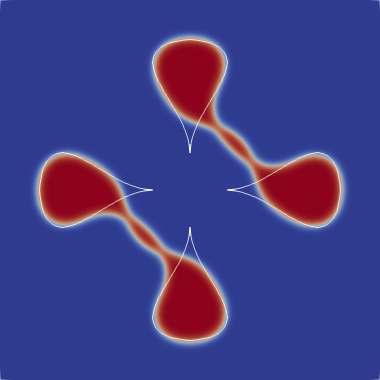}
  \hspace{0.01\textwidth}
  \includegraphics*[width=0.2\textwidth]{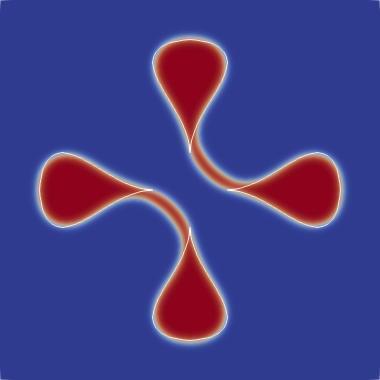}
  \hspace{0.01\textwidth}
  \includegraphics*[width=0.07\textwidth]{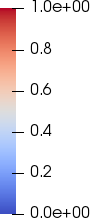}
  \caption{\footnotesize \label{fig:piri2Stand} Evolution of standard flow
    ($\eps=\frac{1}{16}$): Discrete phase-field $\uu_h$ for different times
    $t=0$, $t\approx0.00011$, \mbox{$t\approx0.00159$,} $t\approx0.19510$.}
\end{center}
\end{figure}

\printbibliography

\end{document}

